\documentclass[pdflatex,sn-mathphys-num]{sn-jnl}

\theoremstyle{thmstyleone}%
\newtheorem{theorem}{Theorem}
\newtheorem{proposition}{Proposition}%
\newtheorem{lemma}{Lemma}

\theoremstyle{thmstyletwo}%
\newtheorem{remark}{Remark}%

\theoremstyle{thmstylethree}%
\newtheorem{definition}{Definition}%

\usepackage{graphicx}
\usepackage{url}
\usepackage{xcolor}
\usepackage{etoolbox}    
\usepackage{tablefootnote}
\usepackage{threeparttable}
\usepackage{tablefootnote}
\usepackage{framed,multirow}
\usepackage{amssymb}
\usepackage{amsfonts}
\usepackage{mathrsfs}
\usepackage{upgreek}
\usepackage{epstopdf}
\ifpdf
\DeclareGraphicsExtensions{.eps,.pdf,.png,.jpg}
\else
\DeclareGraphicsExtensions{.eps}
\fi
\usepackage[export]{adjustbox}

\usepackage{amsmath}
\AtBeginDocument{\allowdisplaybreaks[4]}

\usepackage{xspace}
\usepackage{bm}
\usepackage{bbm}

\usepackage{physics}
\usepackage{stmaryrd}
\usepackage{enumitem}
\usepackage[caption=false]{subfig}
\usepackage{mathtools}

\usepackage{booktabs}
\usepackage{multirow, multicol, makecell}
\usepackage{longtable}
\usepackage{threeparttable}

\usepackage{cleveref}

\newtheorem{assumption}{Assumption}

\crefname{assumption}{Assumption}{Assumptions}
\crefname{definition}{Definition}{Definitions}
\crefname{lemma}{Lemma}{Lemmas}
\crefname{remark}{Remark}{Remarks}
\crefname{theorem}{Theorem}{Theorems}
\crefname{proposition}{Proposition}{Propositions}
\crefname{section}{Section}{Sections}
\crefname{figure}{Fig.}{Figs.}
\crefname{equation}{}{}
\crefname{table}{Table}{Tables}
\crefname{appendix}{Appendix}{Appendices}

\renewcommand{\H}[0]{\mathsf{H}}
\newcommand{\D}[0]{\mathsf{D}}
\newcommand{\Q}[0]{\mathsf{Q}}
\newcommand{\E}[0]{\mathsf{E}}

\newcommand{\A}[0]{\mathsf{A}}
\newcommand{\I}[0]{\mathsf{I}}

\newcommand{\J}[0]{\mathsf{J}}

\renewcommand{\P}[0]{\mathsf{P}}
\newcommand{\B}[0]{\mathsf{B}}
\newcommand{\R}[0]{\mathsf{R}}
\newcommand{\N}[0]{\mathsf{N}}

\newcommand{\etal}[0]{{et~al.\@}\xspace}
\newcommand{\eg}[0]{{e.g.\@}\xspace}
\newcommand{\ie}[0]{{i.e.\@}\xspace}

\newcommand{\ignore}[1]{} 

\newcommand{\IR}[1]{\mathbb{R}^{#1}}%

\newcommand{\poly}[1]{\ensuremath{\mathbb{P}^{#1}}({\Omega}_k)}

\newcommand{\der[2]}[]{\dfrac{{\rm d}#1}{{\rm d}#2}}%
\newcommand{\pder[2]}[]{\dfrac{\partial#1}{\partial#2}}
\newcommand{\fn}[1]{\mathcal{#1}}
\newcommand{\fnb}[1]{\bm{\mathcal{#1}}}

\begin{document}

\title[Convergence of entropy-stable continuous summation-by-parts discretizations of symmetric hyperbolic conservation laws]{Convergence of entropy-stable continuous summation-by-parts discretizations of symmetric hyperbolic conservation laws}

\author*[1]{\fnm{Zelalem Arega} \sur{Worku}}
\author[1]{\fnm{David C.} \sur{Del Rey Fern{\'a}ndez}}
\author[2]{\fnm{David W.} \sur{Zingg}}



\abstract{
    The Lax equivalence theorem guarantees convergence of stable and consistent discretizations for linear hyperbolic partial differential equations (PDEs). For nonlinear problems, however, stability and consistency alone do not generally guarantee convergence, even for smooth solutions, and existing convergence results typically rely either on projection-based error decompositions or on linearization arguments that do not directly extend to entropy-stable split-form discretizations. In particular, general convergence results for entropy-stable discretizations of hyperbolic PDEs are currently lacking, despite their widespread use. In this work, we prove convergence under smoothness assumptions on the exact solution and fluxes for entropy-stable split-form discretizations of scalar and symmetric hyperbolic systems with homogeneous flux functions within the continuous summation-by-parts (C-SBP) framework. The scalar inviscid Burgers equation is presented as a canonical example. The analysis is based on a stability-consistency argument that yields a nonlinear error evolution inequality whose solution provides an explicit upper bound on the numerical error. We show that, for sufficiently small mesh spacing, and for degree-$p$ C-SBP discretizations in $d$ spatial dimensions with $p>1+d/2$, this bound remains finite on any finite time interval and tends to zero as the mesh is refined, implying convergence despite the presence of local linear instabilities. The results help clarify the relationship between consistency, entropy stability, nonlinear error growth, and convergence for discretizations of nonlinear hyperbolic problems.
}




\maketitle

\begingroup
\renewcommand\thefootnote{}
\makeatletter
\renewcommand\@makefntext[1]{\noindent\footnotesize #1}
\makeatother
\footnotetext{
\textsuperscript{*} Corresponding author: Zelalem Arega Worku (zelalem.worku@uwaterloo.ca),\\
\hspace*{1em}Emails: David C. Del Rey Fern{\'a}ndez (ddelreyfernandez@uwaterloo.ca), David W. Zingg (david.zingg@utoronto.ca).\\
\textsuperscript{1} Department of Applied Mathematics, University of Waterloo, 200 University Ave. West, Waterloo, ON N2L 3G1, Canada.\\
\textsuperscript{2} Institute for Aerospace Studies, University of Toronto, 4925 Dufferin St, Toronto, ON M3H 5T6, Canada.
}
\endgroup

\section{Introduction}
The Lax equivalence theorem is a cornerstone of numerical discretization methods. It guarantees that a consistent and stable discretization of a linear hyperbolic partial differential equation (PDE) converges to the exact solution under mesh refinement. Extending this result to nonlinear PDEs is of paramount importance, since many physical systems are modeled by nonlinear PDEs. Convergence analysis in this setting, however, is substantially more delicate. To circumvent difficulties associated with nonlinearities and the possible formation of discontinuities, convergence results often assume sufficient smoothness of the exact solution and flux functions. Furthermore, existing convergence analyses typically treat discretizations that are not entropy stable; similar results for entropy-stable schemes remain limited. The purpose of this paper is to extend convergence results under smoothness assumptions to entropy-stable split-form continuous summation-by-parts (C-SBP) discretizations~\cite{hicken2016multidimensional,hicken2020entropy} of scalar and symmetric hyperbolic systems with homogeneous flux functions. 

Early convergence results for nonlinear hyperbolic problems were established using linearization arguments. In particular, Strang~\cite{strang1964accurate} showed that, under sufficient smoothness assumptions on the exact solution and fluxes, consistency and stability of the first variation (linearization) of a nonlinear discretization imply convergence for symmetric hyperbolic systems. The analysis proceeds by linearizing centered finite-difference discretizations about the exact solution and applying von Neumann stability conditions to the resulting linearized system. Given the concerns with linear stability issues for entropy-stable discretizations~\cite{gassner2022stability}, we do not wish to base our convergence results on the stability of the linearized system, although linearization-based arguments may still be applicable under additional structural assumptions, \eg, frozen-coefficient principle \cite{kreiss1989initial,mishra2010stability}, that are not pursued here. 

In the fully discrete discontinuous Galerkin (DG) setting coupled with a second-order explicit Runge--Kutta time discretization, and under sufficient smoothness assumptions, Zhang and Shu~\cite{zhang2004error} derived a priori error estimates for nonlinear scalar conservation laws. These results were subsequently extended to symmetrizable systems of conservation laws in~\cite{zhang2006error}, yielding convergence results for important systems such as the compressible Euler equations. More recently, Huang and Shu~\cite{huang2017error} generalized the scalar convergence analysis of~\cite{zhang2004error} to the case of inexact quadrature. These analyses use globally bounded flux derivatives, obtained by modifying the flux outside the range of the exact solution. Since DG convergence analyses rely on projection-based error decompositions that exploit the polynomial basis structure of the approximation space, they are not directly applicable to general SBP discretizations for which the numerical solution is not explicitly represented via a finite-element basis. Moreover, the existing DG analyses obtain \(L^2\)-stability through interface dissipation furnished by monotone numerical fluxes, and the error estimates are closed by induction arguments that use the smallness of the mesh spacing, whereas the present work seeks to establish convergence by directly exploiting the stability of the SBP discretization.

Entropy-stable and split-form discretizations have gained significant attention in recent years due to their robustness and favorable nonlinear stability properties. Existing analyses of these methods primarily focus on stability or entropy dissipation properties and generally do not address convergence in the sense of vanishing numerical error under mesh refinement. An exception is the contribution of Gassner \etal~\cite{gassner2022stability}, which showed that the linearization of entropy-stable schemes can exhibit local linear instabilities, in the sense that the discrete energy of the linearized scheme may grow faster than that of the corresponding continuous problem, which could affect their convergence. While~\cite{gassner2022stability} suggests that split-form discretizations may nevertheless converge despite this local linear instability, no convergence proof is provided. To address this, the present work develops a convergence analysis for entropy-stable split-form discretizations within the C-SBP framework that do not depend on Strang's~\cite{strang1964accurate} linearization arguments, bypassing the linear stability issues identified in \cite{gassner2022stability}. The analysis is based on a stability--consistency argument that leads to a constant-coefficient Riccati ordinary differential equation (ODE) type error evolution inequality. This inequality provides an explicit upper bound on the numerical error and makes the role of nonlinear error growth transparent. We show that, for sufficiently small mesh spacing, this bound remains finite on a finite time interval and vanishes as the mesh is refined, implying convergence. The resulting error evolution framework clarifies the relationship between consistency, entropy stability, nonlinear error growth, and convergence for discretizations of nonlinear hyperbolic problems and, in line with related long-time error behavior analyses of high-order methods~\cite{nordstrom2008error,kopriva2017error,offner2019error}, may also provide a useful tool for interpreting and comparing nonlinear error growth through the structure and coefficients of the associated error evolution equation.

The results in this paper rely on the energy or entropy stability of the SBP discretization, the consistency of the scheme, smoothness assumptions on the exact solution and flux functions, and the global boundedness of the second derivatives of the flux functions. We restrict our attention to C-SBP discretizations~\cite{hicken2016multidimensional,hicken2020entropy} in order to simplify the analysis and avoid technical complications associated with simultaneous approximation terms (SATs), also known as numerical fluxes in the DG literature. Extending the present analysis to discontinuous SBP discretizations is an important direction for future work. Moreover, we consider only nonlinear scalar conservation laws and symmetric hyperbolic systems with homogeneous flux functions. While symmetry provides a natural energy framework, relatively few physical systems are symmetric, and many important models, such as the compressible Euler equations, are only symmetrizable, and are, consequently, outside the scope of this work. As mentioned earlier, under sufficient smoothness assumptions, a priori error estimates for symmetrizable systems have been derived by Zhang and Shu~\cite{zhang2006error}. It is not straightforward, however, to adapt these error estimates to the semi-discrete setting considered here, since they rely on time-dependent weighted norms involving Jacobian matrices evaluated at elementwise averaged states at each time step. Developing a comparable framework for semi-discrete formulations and extending the present results to symmetrizable systems therefore remains an open problem.

The remainder of the paper is organized as follows. In \cref{sec:preliminaries}, we introduce the C-SBP framework and review the properties of entropy-stable and split-form discretizations. \cref{sec:scalar} presents the convergence analysis for scalar conservation laws. It also presents important properties of the constant-coefficient Riccati ODE relevant for the convergence analysis. In \cref{sec:systems}, the analysis is extended to symmetric hyperbolic systems. Finally, conclusions are presented in \cref{sec:conclusions}.

\section{Preliminaries}\label{sec:preliminaries}
We are interested in understanding the convergence of nonlinear symmetric hyperbolic PDEs with homogeneous flux functions. We consider the following equation on a periodic domain, \(\Omega \subset \IR{d}\),
\begin{equation}\label{eq:nonlinear_prob}
    \pder[\fnb{U}]{t} + \sum_{i=1}^{d}\pder[\fnb{F}_i]{x_{i}} = 0,
\end{equation}
with smooth initial condition, \(\fnb{U}(t=0)\eqqcolon\,\fnb{U}_{0}\) for \(t\in [0,T]\). We assume that both the solution, \(\fnb{U}\in \IR{n_c}\), and the flux, \(\fnb{F}_{i}(\fnb{U})\in \IR{n_c}\), are smooth over the time interval, where \(n_c\) is the number of component equations. In the case of scalar problems, we use script letters without boldface, \eg, \(\fn{U}\), to represent functions. 

Our aim is to show that the error associated with the C-SBP discretization of \cref{eq:nonlinear_prob} vanishes, i.e., the numerical solution converges to the exact solution, under mesh refinement, assuming that a unique solution to the problem exists. To this end, we first introduce C-SBP operators and state important definitions. 

\subsection{Continuous SBP operators}
We assume that the domain \(\Omega\) is compact and connected, and that it is tessellated into \(n_{e}\) non-overlapping elements, \(\fn{T}_{h}\coloneqq \{\{\Omega_{k}\}_{k=1}^{n_e} \colon \Omega=\cup_{k=1}^{n_{e}}\Omega_{k}\}\). The boundaries of each element are assumed to be piecewise smooth and will be referred to as facets
or interfaces, and the union of the facets of element \(\Omega_{k}\) is denoted by \(\Gamma_{k} \coloneqq \partial\Omega_{k}\). We follow the notation of \cite{hicken2020entropy} to introduce local and global C-SBP operators. To move between global and local degrees of freedom we use the restriction operator, \(\P_k\), mapping global indices to the local node ordering on element \(\Omega_k\), and the prolongation matrix, \(\P_k^T\), for the reverse operation.

\paragraph{Local operators.} On element \(\Omega_{k}\), we have the derivative operator in the \(x_{i}\) direction, \(\D_{x_i k}\in \IR{n_p \times n_p}\), that has the decomposition
\begin{equation}
    \D_{x_i k} \,\coloneqq\, \H_k^{-1} \Q_{x_i k},
\end{equation}
where \(n_{p}\) is the number of volume nodes, and \(\H_{k}\) is a symmetric positive definite (SPD) matrix. The \(\Q_{x_i k}\) matrix satisfies the SBP property
\begin{equation}\label{eq:sbp_property}
    \Q_{x_i k} + \Q_{x_i k}^T = \E_{x_i k},
\end{equation}
where \(\E_{x_i k}\) has the structure 
\begin{equation}
    \E_{x_i k} \,\coloneqq\, \sum_{\gamma \in \Gamma_{k}}\R_{\gamma k}^T\B_{\gamma}\N_{x_{i}\gamma k}\R_{\gamma k}.
\end{equation}
Here, \(\R_{\gamma k}\in \IR{n_f\times n_p}\) is an extrapolation matrix from the volume to the facet nodes, the diagonal matrix \(\B_{\gamma}\in \IR{n_f\times n_f}\) contains reference facet quadrature rule weights that are exact for polynomials up to degree \(2p\), \(\N_{x_i \gamma k}\in \IR{n_f\times n_f}\) is a diagonal matrix containing the evaluation of the normal vectors at the facet nodes, and \(n_f\) is the number of facet nodes. In this work, we consider SBP operators with diagonal \(\H\) and diagonal-\(\E\) matrices; hence, the volume and facet nodes are collocated. We note that C-SBP operators are constructed using SBP diagonal-\(\E\) operators \cite{chen2017entropy,hicken2016multidimensional,worku2024quadrature}. 

The SBP derivative operator satisfies the accuracy condition 
\begin{equation}
    [\D_{x_{i}k}\bm{p}_{k}]_{j}= \pder[\fn{P}_{k}]{x_{i}}\left(\bm{x}_{k}^{(j)}\right),
\end{equation}
for \(\fn{P}_k\in \poly{p}\), where \(\bm{p}_{k}\) is a vector containing the evaluation of \(\fn{P}_{k}\) at the volume nodes of \(\Omega_{k}\), \(\bm{x}_{k}^{(j)}\) is the tuple of coordinates of the \(j\)-th volume node of element \(\Omega_{k}\), and \(\poly{p}\) denotes the space of polynomials of degree up to \(p\) on \(\Omega_{k}\). 

\paragraph{Global C-SBP operators.} Using the local-to-global restriction/prolongation operators (see \cite{hicken2020entropy} for details) we assemble global matrices \(\Q_{x_i}\) and \(\H\) that act on the single-valued global numerical solution vector, \(\bm{u}_h\). These global matrices are defined as
\begin{align*}
    \Q_{x_i} \coloneqq \sum_{\Omega_{k} \in \fn{T}_{h}} \P_k^T \Q_{x_i k} \P_k,
    \qquad
    \H \coloneqq \sum_{\Omega_{k} \in \fn{T}_{h}} \P_k^T \H_k \P_k.
\end{align*}
The global differentiation operator in the \(x_{i}\) direction is then computed as
\begin{align}\label{eq:global_D_construction}
    \D_{x_i} &\,\coloneqq\, \H^{-1} \Q_{x_i}.
\end{align}

\paragraph{Periodic cancellation.} For periodic boundary conditions the element boundary contributions cancel pairwise, so the symmetric part of the global \(\Q_{x_i}\) matrix vanishes:
\begin{equation}\label{eq:Q+Q^T=0}
    \Q_{x_i} + \Q_{x_i}^T = \sum_{\Omega_{k} \in \fn{T}_{h}} \P_k^T(\Q_{x_i k} + \Q_{x_i k}^T)\P_k
    = \sum_{\Omega_{k} \in \fn{T}_{h}} \P_k^T \E_{x_i k} \P_k = \mathsf{0}.
\end{equation}
This identity will be used below when deriving global energy identities. Further discussions on SBP operators can be found in the review papers \cite{fernandez2014review,svard2014review}.

\subsection{Error and norm definitions}
For scalar problems, the error on element \(\Omega_{k}\) is defined as 
\begin{equation}
    \bm{e}_{k} \coloneqq \bm{u}_k - \bm{u}_{h,k},
\end{equation}
where \(\bm{u}_{k}\) is a vector containing the restriction of the exact solution at the solution nodes of \(\Omega_{k}\) and \(\bm{u}_{h,k}\) is a vector containing the numerical solution. The global error is defined in a similar manner, 
\begin{equation}\label{eq:error_def}
    \bm{e} \coloneqq \bm{u} - \bm{u}_{h}.
\end{equation}

Since the global norm matrix, \(\H\), is SPD, it defines a norm, \ie, 
\begin{equation}\label{eq:norm_def}
    \|\bm{v}\|_{\H} = \sqrt{\bm{v}^T\H\bm{v}}.
\end{equation} 
Throughout this paper, \(\|\cdot\|_{\H}\) denotes the norm induced by the relevant \(\H\) matrix (global or elementwise, and scalar or block-extended for systems); the intended meaning will be clear from the context.
For scalar problems, we define global norms of the solution and truncation errors by
\begin{equation}\label{eq:global_error_norms}
    \|\bm{e}\|_{\H}^2 \coloneqq \sum_{\Omega_{k}\in \fn{T}_h} \|\bm{e}_{k}\|_{\H}^2, \qquad \|\bm{\tau}\|_{\H}^2 \coloneqq \sum_{\Omega_{k}\in \fn{T}_h} \|\bm{\tau}_{k}\|_{\H}^2.
\end{equation}
For systems of equations, the solution, flux, and error vectors are structured as
\begin{align*}
    \bm{v}           & = [\bm{v}_{1},\dots,\bm{v}_{n_{e}}],
    \tag{vector on all elements}
    \\
    \bm{v}_{k}       & = [\bm{v}_{k}^{(1)},\dots, \bm{v}_{k}^{(n_{p})}], \quad k \in [1,n_{e}],
    \tag{vector on element \(\Omega_{k}\)}
    \\
    \bm{v}_{k}^{(m)} & = [\bm{v}_{k,1}^{(m)},\dots, \bm{v}_{k,n_{c}}^{(m)}], \quad k \in [1,n_{e}], \; m \in [1,n_p]
    \tag{vector on node \(m\)}.
\end{align*}
In this case, the norm for the global error vector is defined by
\begin{equation}\label{eq:error_def_system}
    \|\bm{e}\|_{\H}^2 = \sum_{\Omega_{k}\in\fn{T}_{h}}\sum_{i=1}^{n_c}\|\bm{e}_{k,i}\|_{\H}^2,
    \qquad 
    \|\bm{e}\|_2^2 = \sum_{\Omega_{k}\in\fn{T}_{h}}\sum_{i=1}^{n_c}\|\bm{e}_{k,i}\|_2^2 = \sum_{\Omega_{k}\in\fn{T}_{h}}\|\bm{e}_{k}\|_2^2,
\end{equation}
where the norm of the error for each component on \(\Omega_{k}\) is defined by
\begin{equation}\label{eq:error_elem_def_system}
    \|\bm{e}_{k,i}\|_{\H}^2 = \sum_{m=1}^{n_{p}}w_{k,m}\bm{e}_{k,i}^{(m)}\bm{e}_{k,i}^{(m)} = \bm{e}_{k,i}^T\H_{k}\bm{e}_{k,i},
\end{equation}
and the error on element \(\Omega_{k}\), node \(m\), and the \(i\)-th variable is defined as
\begin{equation}
    \label{eq:error_node_def_system}
    e_{k,i}^{(m)} = u_{k,i}^{(m)} - u_{h,k,i}^{(m)}.
\end{equation}

\subsection{Homogeneity of the flux functions}
An important assumption of this work is homogeneity of the flux functions, which is crucial for establishing discrete entropy stability of the C-SBP discretization of \cref{eq:nonlinear_prob}.

\begin{definition}
     The flux \(\fnb{F}(\fnb{U})\) is a degree \(\beta\) homogeneous function in the argument \(\fnb{U}\) if \(\fnb{F}(\eta \fnb{U})=\eta^{\beta}\fnb{F}(\fnb{U})\), for  \(\eta \neq 0\).
\end{definition}

By Euler's theorem for homogeneous functions, we have the property 
\begin{equation}\label{eq:homogeneity}
    \A\,\fnb{U} = \beta \fnb{F},
\end{equation}
where \(\A \coloneqq \partial\fnb{F}/\partial\fnb{U}\). For the Burgers equation, for instance, the flux, \(\fn{F}(\fn{U}) = \fn{U}^2/2\) is degree \(\beta=2\) homogeneous with respect to \(\fn{U}\).

\section{Entropy stability of the split-form discretization}\label{sec:split_form_stability}
The stability of the numerical solution is imperative for convergence analysis. We consider the split form of \cref{eq:nonlinear_prob} given by
\begin{equation}\label{eq:nonlinear_prob_split}
    \pder[\fnb{U}]{t} + \frac{\beta}{\beta + 1}\sum_{i=1}^{d} \pder[\fnb{F}_i]{x_{i}} + \frac{1}{\beta + 1}\sum_{i=1}^{d}\A_{i}\pder[\fnb{U}]{x_{i}} = 0,
\end{equation} 
which, when discretized using C-SBP operators, results in a stability proof of the semi-discrete scheme that is analogous to the continuous proof. We note that \(\beta\) is the degree of homogeneity of the flux with respect to the conservative variable. If the conservative variables are not the same as the entropy variables, a similar homogeneity property with respect to entropy variables can also be used for the splitting, \eg, see \cite{olsson1994energy,gerritsen1996designing,yee2000entropy,sjogreen2019entropy,worku2023entropy}; however, we do not consider those cases in this paper. As an example, for the Burgers equation, the conservative variable is also the entropy variable. Given, \(\beta=2\) and \(\fn{F}=\fn{U}^2/2\) for the Burgers equation, the splitting in \cref{eq:nonlinear_prob_split} simplifies to the widely used form 
\begin{equation}\label{eq:nonlinear_prob_split_burgers}
    \pder[\fn{U}]{t} + \sum_{i=1}^{d}\frac{1}{3} \pder[\fn{U}^2]{x_{i}} + \frac{1}{3}\fn{U}\pder[\fn{U}]{x_{i}} = 0.
\end{equation} 
Since the scalar split form is a special case of the split form for systems of equations, \cref{eq:nonlinear_prob_split}, we only show entropy stability of the latter. For the symmetric and homogeneous class considered here, we can take a quadratic entropy function in \(\fnb{U}\); hence, entropy stability is equivalent to an \(L^2\) energy bound on \(\fnb{U}\). 

\subsection{Continuous stability analysis}
To study the stability of \cref{eq:nonlinear_prob_split}, we premultiply by \(\fnb{U}^T\) and integrate over the domain, 
\begin{align}\label{eq:energy_0}
    \int_{\Omega}\fnb{U}^T\pder[\fnb{U}]{t} \dd{\Omega}= - \frac{\beta}{\beta + 1} \sum_{i=1}^{d}\int_{\Omega}\fnb{U}^T\pder[\fnb{F}_i]{x_{i}} \dd{\Omega}-  \frac{1}{\beta + 1}\sum_{i=1}^{d}\int_{\Omega}\fnb{U}^T\A_{i}\pder[\fnb{U}]{x_{i}}\dd{\Omega}.
\end{align}
Simplifying the temporal term, using the homogeneity property, \cref{eq:homogeneity}, the symmetry of \(\A_{i}\), and applying integration by parts to the last term on the right-hand side (RHS) of \cref{eq:energy_0} gives, 
\begin{align}\label{eq:energy_1}
    \frac{1}{2}\der[]{t}\|\fnb{U}\|_{2}^2 &=  - \frac{\beta}{\beta+1} \sum_{i=1}^{d}\int_{\Gamma}\fnb{F}_{i}^T \fnb{U} \,n_{i}\dd{\Gamma}
     = 0,
\end{align}
where the boundary term vanishes due to the periodic boundary conditions. Integrating in time, we obtain the energy conservation statement, 
\begin{equation}
    \|\fnb{U}(T)\|_{2}^{2} = \|\fnb{U}_{0}\|_{2}^{2}.
\end{equation}

\subsection{Semi-discrete entropy-stability analysis}
The split-form semi-discretization of \cref{eq:nonlinear_prob_split} with C-SBP operators is written as 
\begin{equation}\label{eq:split_form_disc}
    \der[\bm{u}_{h}]{t} + \frac{\beta}{\beta + 1}\sum_{i=1}^{d}\overline{\D}_{x_i}\bm{f}_{i}(\bm{u}_{h}) + \frac{1}{\beta + 1}\sum_{i=1}^{d}\bar{\A}_{i}\overline{\D}_{x_i}\bm{u}_{h} = 0.
\end{equation}
where \(\bar{\A}_{i}\) is a symmetric block diagonal matrix containing \(\A_{i}(\bm{u}^{(j)}_{h})\in\IR{n_c\times n_c}\) at its \(j\)-th diagonal block associated with the \(j\)-th global node and \(\overline{\D}_{x_i}\) is constructed using \(\Q_{x_i k}\otimes\I_{n_{c}}\), \(\H_{k}\otimes\I_{n_{c}}\), and \(\P_{k}\otimes\I_{n_{c}}\) as in \cref{eq:global_D_construction}. Here, \(\otimes\) denotes the Kronecker product and \(\I_{n_c} \in \IR{n_c\times n_c}\) is the identity matrix. Following the continuous stability analysis, we premultiply \cref{eq:split_form_disc} by \(\bm{u}_{h}^T\overline{\H}\) and simplify to find, 
\begin{align}\label{eq:energy_1_disc}
    \bm{u}_{h}^T\overline{\H}\der[\bm{u}_{h}]{t} &= -\frac{\beta}{\beta + 1}\sum_{i=1}^{d}\bm{u}_{h}^T\overline{\Q}_{x_i}\bm{f}_{i}(\bm{u}_{h}) - \frac{1}{\beta + 1}\sum_{i=1}^{d}\bm{u}_{h}^T\overline{\H}\bar{\A}_{i}\overline{\D}_{x_i}\bm{u}_{h}.
\end{align}
Due to the block diagonal structure of \(\bar{\A}_{i}\) and the application of the Kronecker product in constructing the global diagonal \(\overline{\H}\) matrix, we have the commutativity property,
\begin{equation}\label{eq:HA=AH}
    \overline{\H} \bar{\A}_{i} = \bar{\A}_{i}\overline{\H}. 
\end{equation} 
Applying \cref{eq:HA=AH} in \cref{eq:energy_1_disc}, using the homogeneity identity, \cref{eq:homogeneity}, to write \(\bm{u}_{h}^T\bar{\A}_{i} = \beta \bm{f}_{i}^T\), and using the SBP property, we find 
\begin{align}\label{eq:energy_2_disc}
    \bm{u}_{h}^T\overline{\H}\der[\bm{u}_{h}]{t} &= -\frac{\beta}{\beta + 1}\sum_{i=1}^{d}\bm{u}_{h}^T\overline{\Q}_{x_i}\bm{f}_{i}(\bm{u}_{h}) - \frac{\beta}{\beta + 1}\sum_{i=1}^{d}\bm{f}_i^T(\bm{u}_{h})\overline{\Q}_{x_i}\bm{u}_{h},
    \nonumber
    \\
    \frac{1}{2}\der[]{t}\|\bm{u}_{h}\|_{\H}^2 &= - \frac{\beta}{\beta + 1}\sum_{i=1}^{d}\bm{u}_{h}^T\overline{\E}_{x_i} \bm{f}_i(\bm{u}_{h})
    = 0,
\end{align}
where we have used the SBP property, \cref{eq:sbp_property}, to arrive at the third line and applied the periodicity of the problem to eliminate the resulting boundary term. Hence, as in the continuous case, upon integration over the time interval, we find discrete energy conservation,
\begin{equation}\label{eq:solution_bound}
    \|\bm{u}_{h}(T)\|_{\H}^{2} = \|\bm{u}_{0}\|_{\H}^{2},
\end{equation}
where \(\bm{u}_{0}\) is the restriction of the initial condition at the solution nodes. This guarantees that the numerical solution of the split-form discretization, \cref{eq:split_form_disc}, is bounded until the final time, \(T\).

\section{Convergence analysis for nonlinear scalar conservation laws}\label{sec:scalar}
In this section, we study convergence of C-SBP discretizations for nonlinear scalar conservation laws. For notational simplicity, we present the analysis in a single spatial direction and suppress the index \(i\) in the derivative operators and discrete vectors; the multidimensional case follows by repeating the same estimates in each coordinate direction and summing the resulting bounds over \(i=1,\dots,d\). The effect of this on the final form of the error equation is to increase the coefficients in the Riccati ODE, \(a\), \(b\), and \(c\), which are defined in the next section. To relate these coefficients to mesh refinement, we assume the following standard quasi-uniformity conditions.
\begin{assumption}\label{assump:mesh}
Let \(\{\fn{T}_h\}_{h\to 0}\) be a family of meshes of \(\Omega\). For each \(\fn{T}_h\) and each element \(\Omega_k\in\fn{T}_h\), let \(\H_k=\mathrm{diag}(w_{k,1},\dots,w_{k,n_p})\) with \(w_{k,m}>0\), and define
\[
w_{\min}\,\coloneqq\, \min_{\Omega_{k}\in\fn{T}_h}\min_{m\in[1,n_p]} w_{k,m},\qquad
w_{\max}\,\coloneqq\, \max_{\Omega_{k}\in\fn{T}_h}\max_{m\in[1,n_p]} w_{k,m},\qquad
h \,\coloneqq\, \left({|\Omega|}/{n_e}\right)^{1/d},
\]
where \(|\Omega|\) is the volume of the domain and \(h\) is the nominal mesh size. Assume the family is quasi-uniform in the sense that there exist constants \(0<c\le C\) that are uniform over the family, \ie, independent of \(h\), such that
\begin{equation}\label{eq:bound_on_H_ii}
c\,h^{d} \le w_{k,m}\le C\,h^{d},\qquad \forall\,\Omega_k\in\fn{T}_h,\ \forall\,m.
\end{equation}
In particular, \(w_{\max}/w_{\min}\le C/c=\fn{O}(1)\) as \(h\to 0\).
\end{assumption}

We next consider the split-form discretization of \cref{eq:nonlinear_prob},
\begin{equation}\label{eq:csbp_disc_split_form}
    \der[\bm{u}_{h}]{t}
    = -\alpha_{1}\D f(\bm{u}_{h})
    - \alpha_{2}\A(\bm{u}_{h})\D \bm{u}_{h},
\end{equation}
where \(\A(\bm{u}_{h})\) is a diagonal matrix containing the evaluation \(\partial{\fn{F}}/\partial{\fn{U}}\) at \(\bm{u}_{h}\), and \(\alpha_{1}\) and \(\alpha_{2}\) are fixed splitting coefficients determined by the splitting procedure described in \cref{sec:split_form_stability} and coincide with the coefficients of the corresponding continuous split form.

Subtracting \cref{eq:csbp_disc_split_form} from the substitution of the exact solution into \cref{eq:csbp_disc_split_form}, we obtain the error equation
\begin{equation}\label{eq:error_split_form}
    \der[\bm{e}]{t}
    = -\alpha_{1}\D (f(\bm{u}) - f(\bm{u}_{h}))
    - \alpha_{2} \bigl(\A(\bm{u})\D\bm{u} - \A(\bm{u}_{h})\D\bm{u}_{h}\bigr)
    + \bm{\tau},
\end{equation}
where \(\bm{\tau}\) is the truncation error vector defined as
\begin{equation}\label{eq:tau_csbp_split_form}
    \bm{\tau}\coloneqq \der[\bm{u}]{t} + \alpha_{1}\D f(\bm{u}) + \alpha_{2}\A(\bm{u})\D \bm{u} = \fn{O}(h^p).
\end{equation}
Since the flux is smooth with globally bounded second derivative by assumption and the numerical solution is bounded, \ie, \(\|\bm{u}_{h}(t)\|_{\H}< \infty\) for \(t\in[0,T]\), a first-order Taylor expansion gives (see \cref{app:remainder_bound})
\begin{equation}\label{eq:flux_jump_taylor}
    f(\bm{u})-f(\bm{u}_h) = \A(\bm{u}) \bm{e} - R_1(\bm{e}),
\end{equation}
where \(R_1(\bm{e})\) is a remainder term, which is second-order in \(\bm{e}\). Substituting \cref{eq:flux_jump_taylor} into \cref{eq:error_split_form}, premultiplying the result by \(\bm{e}^T\H\), and simplifying, we arrive at the error equation
\begin{align}\label{eq:energy_error_eq_nonlinear_split0}
\frac{1}{2}\der[]{t}\|\bm{e}\|_{\H}^2
&=
-\alpha_{1}\underbrace{\bm{e}^T\Q\A(\bm{u})\bm{e}}_{\coloneqq\text{Term I}}
+\alpha_{1}\underbrace{\bm{e}^T\Q R_1(\bm{e})}_{\coloneqq\text{Term II}} 
-\alpha_{2}\underbrace{\bigl(\bm{e}^T\A(\bm{u})\Q\bm{u} - \bm{e}^T\A(\bm{u}_{h})\Q\bm{u}_{h}\bigr)}_{\coloneqq\text{Term III}}
+\underbrace{\bm{e}^T\H\bm{\tau}}_{\coloneqq\text{Term IV}}.
\end{align}
where we have used the commutativity of diagonal matrices, \ie, \(\H\A = \A\H\). Next, we derive separate bounds for Terms I-IV on the right-hand side of \cref{eq:energy_error_eq_nonlinear_split0}.

\subsection{Bound for Term I}
The first term on the RHS of \cref{eq:energy_error_eq_nonlinear_split0} can be written as
\begin{align}
    \bm{e}^T\Q\A(\bm{u})\bm{e} & = \bm{e}^T\Big(\sum_{\Omega_k\in\fn{T}_{h}}\P_{k}^T \Q_{k}\P_{k}\Big)\A(\bm{u})\bm{e}
     = \sum_{\Omega_k\in\fn{T}_{h}}\bm{e}_{k}^T\Q_{k}\A(\bm{u}_{k})\bm{e}_{k}
    \nonumber
    \\
                               & = \frac{1}{2}\sum_{\Omega_k\in\fn{T}_{h}}\bm{e}_{k}^T
    \big(\Q_{k}\A(\bm{u}_{k}) + (\Q_{k}\A(\bm{u}_{k}))^T\big)\bm{e}_{k}.
\end{align}
Since \(\A_{k} = \A_{k}^T\) and \(\Q_{k}^T = -\Q_{k} + \E_{k}\), we find
\begin{align}
    \bm{e}^T\Q\A(\bm{u})\bm{e}
     & = \frac{1}{2}\sum_{\Omega_k\in\fn{T}_{h}}\bm{e}_{k}^T
    \big(\Q_{k}\A(\bm{u}_{k}) - \A(\bm{u}_{k}) \Q_{k} \big)\bm{e}_{k} + \frac{1}{2}\sum_{\Omega_k\in\fn{T}_{h}}\bm{e}_{k}^T \E_{k}\A(\bm{u}_{k})\bm{e}_{k}
    \nonumber
    \\
     & =\frac{1}{2}\sum_{\Omega_k\in\fn{T}_{h}}\bm{e}_{k}^T
    \big(\Q_{k}\A(\bm{u}_{k}) - \A(\bm{u}_{k}) \Q_{k} \big)\bm{e}_{k},
\end{align}
where the last term in the first equality vanishes because \(\E_{k}\) is diagonal by assumption (and hence, at the collocated facet and volume nodes, the contributions from neighboring elements cancel exactly when summing over all elements).

We note that
\[
    (\Q_{k} \A_{k} - \A_{k}\Q_{k})_{ij} = [\Q_{k}]_{ij} \; (f_{k,j}' - f_{k,i}'),
\]
where \(f'_{k,i}=\frac{\partial f_{k}}{\partial u}(u_{i})\).
Now, define a skew-symmetric flux Jacobian difference matrix, \(\A^*\), as
\[
    [\A^*_{k}]_{ij} \coloneqq f_{k,j}' - f_{k,i}'.
\]
Then, we have
\[
    \Q_{k} \A(\bm{u}_{k}) - \A(\bm{u}_{k})\Q_{k} = \Q_{k} \circ \A^{*}_{k},
\]
where \(\circ\) denotes the Hadamard (elementwise) product. Consequently, we can write
\begin{equation}\label{eq:hadamard_QF}
    \bm{e}^T\Q\A\bm{e} = \frac{1}{2}\sum_{\Omega_k\in\fn{T}_{h}}\bm{e}^T_{k}(\Q_{k} \circ \A_{k}^{*})\bm{e}_{k}.
\end{equation}
In the next lemma, we show that the left hand side (LHS) of \cref{eq:hadamard_QF}, \ie, Term I, is bounded by a constant times the \(L^2\) norm of the error.
\begin{lemma}\label{lem:hadamard_bound}
    Let \(\bm{e}_{k} \in \mathbb{R}^{n_{p}}\), \(\Q_{k} \in \mathbb{R}^{n_{p}\times n_{p}}\), \(\A^{*}_{k} \in \mathbb{R}^{n_{p}\times n_{p}}\) be a skew-symmetric flux Jacobian difference matrix, and define
    \begin{align}
        |\A^{*}_{k}|_{\infty} &\coloneqq \max_{i,j}|[\A^{*}_k]_{ij}| \text{ with} \;i,j \in \{1,\dots,n_{p}\},
        \label{eq:def_f*k}
        \\
        |\A^{*}|_{\infty} &\coloneqq \max_{\Omega_{k}\in \fn{T}_{h}}|\A^{*}_{k}|_{\infty},
        \label{eq:def_f*_inf}
        \\
        \|\Q_{*}\|_{2} &\coloneqq \max_{\Omega_{k}\in \fn{T}_{h}}\|\Q_{k}\|_{2},
        \label{eq:def_Q*}
        \\
        c_{F} &\coloneqq \frac{n_{p}^{3/2}}{2} |\A^{*}|_{\infty}\,\|\Q_*\|_{2}
        \label{eq:c_F}.
    \end{align}
    Then, we have 
    \begin{equation}\label{eq:term1_bound}
        |\bm{e}^T\Q\A\bm{e}| \le \frac{1}{2}\sum_{\Omega_k\in\fn{T}_{h}}|\bm{e}_{k}^{T}(\A_{k}^{*}\circ\Q_{k})\bm{e}_{k}|
        \le \sum_{\Omega_k\in\fn{T}_{h}}\frac{n_{p}^{3/2}}{2}|\A^{*}_{k}|_{\infty}\,\|\Q_{k}\|_{2}\,\|\bm e_{k}\|_{2}^{2}
        \le c_{F}\|\bm e\|_{2}^{2}.
    \end{equation}
\end{lemma}

\begin{proof}
    First, we focus on a single element contribution. We have
    \[
        |\bm{e}_{k}^T(\A^{*}_{k}\circ \Q_{k})\bm{e}_{k}|\le \|{\A}^{*}_{k}\circ \Q_{k}\|_{2}\|\bm{e}_{k}\|_{2}^2.
    \]
    Using the submultiplicativity of the 2-norm for Hadamard products, \eg, see \cite{johnson1990matrix,bhatia1996matrix}, we write
    \[
        |\bm{e}_{k}^T(\A^{*}_{k}\circ \Q_{k})\bm{e}_{k}|
        \le \|{\A}^{*}_{k}\|_{2}\|\Q_{k}\|_{2}\|\bm{e}_{k}\|_{2}^2.
    \]
    For a matrix \(\A\in\IR{n_p\times n_p}\), using the norm relations \(\|\A\|_{2} \le \sqrt{n_p}\|\A\|_{\infty}\) and \(\|\A\|_{\infty} \le n_{p} \max_{i,j}|\A_{ij}| \), we obtain
    \[
        \|\A_{k}^*\|_{2} \le n_{p}^{3/2}\max_{i,j}|[\A^*_{k}]_{i,j}| = n_{p}^{3/2}|\A^*_{k}|_{\infty}.
    \]
    Substituting, we obtain
    \[
        |\bm{e}_{k}^T(\A^{*}_{k}\circ \Q_{k})\bm{e}_{k}|
        \le n_{p}^{3/2}|\A^*_{k}|_{\infty} \|\Q_{k}\|_{2}\|\bm{e}_{k}\|_{2}^2.
    \]
    Summing over all elements, using the triangle inequality, and the definitions of \(\|\Q_{*}\|_{2}\) and \(|\A^{*}|_{\infty}\) along with \cref{eq:global_error_norms}, we arrive at
    \begin{align*}
        |\bm{e}^T\Q\A\bm{e}| \le \frac{1}{2}\sum_{\Omega_k\in\fn{T}_{h}}|\bm{e}_{k}^T(\A^{*}_{k}\circ \Q_{k})\bm{e}_{k}|
         & \le \sum_{\Omega_k\in\fn{T}_{h}}\frac{n_{p}^{3/2}}{2} |\A^*_{k}|_{\infty} \|\Q_{k}\|_{2}\|\bm{e}_{k}\|_{2}^2
        \\
         & \le\frac{n_{p}^{3/2}}{2} |\A^*|_{\infty} \|\Q_{*}\|_{2}\|\bm{e}\|_{2}^2 
         = c_{F}\|\bm e\|_{2}^{2},
    \end{align*}
    as desired. 
\end{proof}


\subsection{Bound for Term II}
The second term on the RHS of \cref{eq:energy_error_eq_nonlinear_split0} can be written as
\begin{align}
    |\bm{e}^T\Q R_1(\bm{e})| & = \Big|\bm{e}^T\Big(\sum_{\Omega_{k}\in \fn{T}_{h}}\P_{k}^T\Q_{k}\P_{k} \Big)R_1(\bm{e})\Big| =\Big|\sum_{\Omega_{k}\in \fn{T}_{h}}\bm{e}_{k}^T\Q_{k}R_1(\bm{e}_{k})\Big|
    \nonumber
    \\
                           & \le \sum_{\Omega_{k}\in \fn{T}_{h}}|\bm{e}_{k}^T\Q_{k}R_1(\bm{e}_{k})|
                             \le \sum_{\Omega_{k}\in \fn{T}_{h}}\|\bm{e}_{k}\|_{2}\|\Q_{k}R_1(\bm{e}_{k})\|_{2}
    \nonumber
    \\
                           & \le \Big(\sum_{\Omega_{k}\in \fn{T}_{h}}\|\bm{e}_{k}\|_{2}^2\Big)^{1/2} \Big(\sum_{\Omega_{k}\in \fn{T}_{h}}\|\Q_{k}R_1(\bm{e}_{k})\|_{2}^2\Big)^{1/2}
    \tag{By \cref{lem:sum_of_H_norm_products}}
    \nonumber
    \\
                           & = \|\bm{e}\|_{2} \Big(\sum_{\Omega_{k}\in \fn{T}_{h}}\|\Q_{k}R_1(\bm{e}_{k})\|_{2}^2\Big)^{1/2},\label{eq:eQR0}
\end{align}
where we have used the global error definition in \cref{eq:global_error_norms} in the last line. We can simplify \cref{eq:eQR0} as
\begin{align}
    \|\bm{e}\|_{2} \Big(\sum_{\Omega_{k}\in \fn{T}_{h}}\|\Q_{k}R_1(\bm{e}_{k})\|_{2}^2\Big)^{1/2} \le
    \|\bm{e}\|_{2} \Big(\sum_{\Omega_{k}\in \fn{T}_{h}}\|\Q_{k}\|_{2}^2\|R_1(\bm{e}_{k})\|_{2}^2\Big)^{1/2}. \label{eq:eQR1}
\end{align}
Using the definition of \(\|\Q_{*}\|_{2}\) in \cref{eq:def_Q*}, we write \cref{eq:eQR1} as
\begin{align}
    |\bm{e}^T\Q R_1(\bm{e})| & \le \|\bm{e}\|_{2} \Big(\|\Q_{*}\|_{2}^2\sum_{\Omega_{k}\in \fn{T}_{h}}\|R_1(\bm{e}_{k})\|_{2}^2\Big)^{1/2}
    \le \|\bm{e}\|_{2} \|\Q_{*}\|_{2}\Big(\sum_{\Omega_{k}\in \fn{T}_{h}}\|R_1(\bm{e}_{k})\|_{2}^2\Big)^{1/2}
    \nonumber
    \\
                           & \le \frac{c_{r}}{2}\|\Q_{*}\|_{2}\|\bm{e}\|_{2}^3,
    \label{eq:eQR2}
\end{align}
where we have used the bound \(\|R_1(\bm{e})\|_{2} \le \frac{c_{r}}{2}\|\bm{e}\|_{2}^2\) in \cref{eq:remainder_2_norm_bound1} in the third inequality. The coefficient \(c_r\) is defined by \(c_r \coloneqq \max_{i} |f''(u_{\xi,i})|\) for each global node \(i\), where \(u_{\xi,i}\in [u_i,u_{h,i}]\) is the intermediate value from Taylor's theorem and \(f''={\partial}^2 f/\partial u^2\).
Introducing the definition,
\begin{equation}\label{eq:c_R}
    c_{R}\coloneqq \frac{c_{r}}{2}\|\Q_*\|_{2},
\end{equation}
and using \cref{eq:eQR2}, we write the bound for Term II as
\begin{equation}\label{eq:term2_bound}
    |\bm{e}^T\Q R_1(\bm{e})| \le c_{R}\|\bm{e}\|_{2}^3.
\end{equation}

\subsection{Bound for Term III}
To derive a bound on the third term in \cref{eq:energy_error_eq_nonlinear_split0}, we first expand \(\A(\bm u_h)\) (the diagonal Jacobian evaluated at \(\bm u_h=\bm u-\bm e\)) by a pointwise Taylor expansion of \(f'\) about \(u_i\); for each node \(i\),
\[
    f'(u_{h,i}) = f'(u_i - e_i)
    = f'(u_i) - f''(u_{\xi,i})\,e_i
\]
for some \(u_{\xi,i}\in [u_{i},u_{h,i}]\). In matrix form, we have
\begin{equation}\label{eq:F_expand}
    \A(\bm{u}_h) = \A(\bm{u}) - \R_{0}(\bm{e}) ,
\end{equation}
where the linear remainder term, \(\R_{0}(\bm{e})\), is a diagonal matrix with entries \(|[\R_{0}(\bm e)]_{ii}|\le {c}_{r} |e_i|\) with \({c}_{r}=\max_{i}|f''(u_{\xi,i})|\) for each global node \(i\). Using the error definition, \(\bm{e}=\bm{u} - \bm{u}_{h}\), and \cref{eq:F_expand}, we can write Term III as 
\begin{align}
    \bm{e}^T\A(\bm{u})\Q\bm{u} - \bm{e}^T\A(\bm{u}_{h})\Q\bm{u}_{h}
    & = \bm{e}^T\A(\bm{u})\Q\bm{u}
    - \bm{e}^T(\A(\bm{u}) - \R_{0}(\bm{e}))\Q(\bm{u} - \bm{e})
    \nonumber
    \\
    & =\bm{e}^T\A(\bm{u})\Q\bm{e}
    + \bm{e}^T\R_{0}(\bm{e})\Q\bm{u}
    - \bm{e}^T\R_{0}(\bm{e})\Q\bm{e}
    \nonumber
    \\
    & =\bm{e}^T\A(\bm{u})\Q\bm{e}
    + \bm{e}^T\R_{0}(\bm{e})\H\D\bm{u}
    - \bm{e}^T\R_{0}(\bm{e})\Q\bm{e}
    \nonumber
    \\
    & =\bm{e}^T\A(\bm{u})\Q\bm{e}
    + \bm{e}^T\R_{0}(\bm{e})\H\bm{u}_{x}
    + \bm{e}^T\R_{0}(\bm{e})\H\bm{\tau}_{u}
    - \bm{e}^T\R_{0}(\bm{e})\Q\bm{e} ,\label{eq:phi1}
\end{align}
where we have used the accuracy of the derivative operators to write
\[
    [\D\bm{u}]_{i} =\pder[\fn{U}]{x}(x_{i}) + \fn{O}(h^{p})
    = [\bm{u}_{x}]_{i} + [\bm{\tau}_{u}]_{i}.
\]

We now analyze each term in \cref{eq:phi1}. The first term leads to a similar bound as in \cref{eq:term1_bound},
\begin{equation}\label{eq:trace_bound2}
    |\bm{e}^T\A(\bm{u})\Q\bm{e}|
    \le \sum_{\Omega_k\in\fn{T}_{h}}\frac{n_{p}^{3/2}}{2}|\A^{*}|_{\infty}\,\|\Q_{k}\|_{2}\,\|\bm e_{k}\|_{2}^{2}
    \le \,c_{F}\|\bm e\|_{2}^{2}.
\end{equation}
The second and third terms are of the same type; hence, we use \(\bm{v}\) instead of \(\bm{u}_{x}\) and \(\bm{\tau}_{u}\) for their analysis and proceed as
\begin{align}
    |\bm{e}^T\R_{0}(\bm{e})\H\bm{v}|
    & = \Big|\sum_{\Omega_{k}\in\fn{T}_{h}}\bm{e}_{k}^T\R_{0}(\bm{e}_{k})\H_{k}\bm{v}_{k}\Big|
    \nonumber
    \\
    & \le \sum_{\Omega_{k}\in\fn{T}_{h}}|\bm{e}_{k}^T\R_{0}(\bm{e}_{k})\H_{k}\bm{v}_{k}|
    \le \sum_{\Omega_{k}\in\fn{T}_{h}}\|\bm{e}_{k}^T\R_{0}(\bm{e}_{k})\|_{\H}\|\bm{v}_{k}\|_{\H}
    \nonumber
    \\
    & \le  \sqrt{w_{\max}}\sum_{\Omega_{k}\in\fn{T}_{h}}\|\bm{e}_{k}^T\R_{0}(\bm{e}_{k})\|_{\H}\|\bm{v}_{k}\|_{2}
    \tag{\(\|\bm{v}\|_{\H} \le \sqrt{w_{\max}}\|\bm{v}\|_{2}\)}
    \nonumber
    \\
    & \le  \sqrt{w_{\max}}\|\bm{v}_{*}\|_{2}\sum_{\Omega_{k}\in\fn{T}_{h}}\|\bm{e}_{k}^T\R_{0}(\bm{e}_{k})\|_{\H}
    \tag{\(\|\bm{v}_{*}\|_{2}\coloneqq \max_{\Omega_{k}} \|\bm{v}_{k}\|_{2}\)}
    \nonumber
    \\
    & \le  \sqrt{w_{\max}}\sqrt{n_p}\|\bm{v}_{*}\|_{\infty}\sum_{\Omega_{k}\in\fn{T}_{h}}\|\bm{e}_{k}^T\R_{0}(\bm{e}_{k})\|_{\H}
    \tag{\(\|\bm{v}_{*}\|_{2}\le \sqrt{n_p}\|\bm{v}_{*}\|_{\infty}\)}
    \nonumber
    \\
    |\bm{e}^T\R_{0}(\bm{e})\H\bm{v}|
    & \le  \sqrt{w_{\max}}\sqrt{n_p}\|\bm{v}_{*}\|_{\infty}\sum_{\Omega_{k}\in\fn{T}_{h}}\|\bm{e}_{k}^T\R_{0}(\bm{e}_{k})\|_{\H}. \label{eq:term23}
\end{align}
Then we note that
\begin{align}
    \|\bm{e}_{k}^T\R_{0}(\bm{e}_{k})\|_{\H}^2
    & \le {c}_{r}^2\sum_{i=1}^{n_{p}} w_{k,i}e_{k,i}^2 e_{k,i}^2
    \le {c}_{r}^2\sum_{i=1}^{n_{p}} (\max_{1\le j\le n_p}e_{k,j}^2) (w_{k,i}e_{k,i}^2)
    = {c}_{r}^2\|\bm{e}_{k}\|_{\infty}^2 \|\bm{e}_{k}\|_{\H}^2
    \nonumber
    \\
    \|\bm{e}^T_{k}\R_{0}(\bm{e}_{k})\|_{\H}
    & \le  {c}_{r}\|\bm{e}_{k}\|_{\infty} \|\bm{e}_{k}\|_{\H}. \label{eq:inf_H_relation}
\end{align}
Furthermore, we have
\begin{align}
    \|\bm{e}_{k}\|_{\H}^2
    = \sum_{i=1}^{n_{p}}w_{k,i}e_{k,i}^2
    \ge w_{\min}\sum_{i=1}^{n_{p}}e_{k,i}^2
    \ge w_{\min} \max_{i}\bm{e}_{k,i}^2
    = w_{\min} \|\bm{e}_{k}\|_{\infty}^2;\nonumber
\end{align}
thus, we obtain
\begin{equation}
    \|\bm{e}_{k}\|_{\infty} \le \frac{1}{\sqrt{w_{\min}}}\|\bm{e}_{k}\|_{\H}.
\end{equation}
Therefore, substituting into \cref{eq:inf_H_relation} and summing over all elements yields
\begin{equation}\label{eq:eMH}
    \sum_{\Omega_{k}\in\fn{T}_{h}} \|\bm{e}_{k}^T\R_{0}(\bm{e}_{k})\|_{\H} 
    \le {c}_{r}\frac{1}{\sqrt{w_{\min}}}\sum_{\Omega_{k}\in\fn{T}_{h}}\|\bm{e}_{k}\|_{\H}^2
    = {c}_{r}\frac{1}{\sqrt{w_{\min}}}\|\bm{e}\|_{\H}^2.
\end{equation}
Substituting \cref{eq:eMH} into \cref{eq:term23}, we obtain,
\begin{align}\label{eq:term23_1}
    |\bm{e}^T\R_{0}(\bm{e})\H\bm{v}|
    &\le  {c}_{r}\frac{\sqrt{w_{\max}}}{\sqrt{w_{\min}}}\sqrt{n_p}\|\bm{v}_{*}\|_{\infty} \|\bm{e}\|_{\H}^2
    \le
     {c}_{r}\,w_{\max}\frac{\sqrt{w_{\max}}}{\sqrt{w_{\min}}}\sqrt{n_p}\|\bm{v}_{*}\|_{\infty}  \|\bm{e}\|_{2}^2.
\end{align}
Therefore, we can bound the second and third terms in \cref{eq:phi1} as
\begin{align}
    |\bm{e}^T\R_{0}(\bm{e})\H\bm{u}_{x}|
    &\le  {c}_{r}\,w_{\max}\frac{\sqrt{w_{\max}}}{\sqrt{w_{\min}}}\sqrt{n_p}\|\bm{u}_{x,*}\|_{\infty} \|\bm{e}\|_{2}^2, \label{eq:eMHux}
    \\
    |\bm{e}^T\R_{0}(\bm{e})\H\bm{\tau}_{u}|
    &\le  {c}_{r}\,w_{\max}\frac{\sqrt{w_{\max}}}{\sqrt{w_{\min}}}\sqrt{n_p}\|\bm{\tau}_{u,*}\|_{\infty}\|\bm{e}\|_{2}^2, \label{eq:eMHtau_u}
\end{align}
which are of the same form as the bounds given in \cref{eq:trace_bound2} and \cref{eq:term1_bound}. 

Finally, the fourth term on the RHS of \cref{eq:phi1} can be analyzed as follows,
\begin{align}
    |\bm{e}^T\R_{0}(\bm{e})\Q\bm{e}|
    & = \Big|\sum_{\Omega_{k}\in \fn{T}_{h}}\bm{e}_{k}^T\R_{0}(\bm{e}_{k})\Q_{k}\bm{e}_{k}\Big|
    \nonumber
    \\
    & \le \sum_{\Omega_{k}\in \fn{T}_{h}}|\bm{e}_{k}^T\R_{0}(\bm{e}_{k})\Q_{k}\bm{e}_{k}|
    \le \sum_{\Omega_{k}\in \fn{T}_{h}}\|\bm{e}_{k}^T\R_{0}(\bm{e}_{k})\|_{2}\|\Q_{k}\bm{e}_{k}\|_{2}
    \nonumber
    \\
    & \le  {c}_{r}\sum_{\Omega_{k}\in \fn{T}_{h}}\Big(\sum_{i=1}^{n_{p}}e_{k,i}^4\Big)^{1/2}\|\Q_{k}\|_{2}\|\bm{e}_{k}\|_{2}
    \nonumber
    \\
    & \le  {c}_{r}\sum_{\Omega_{k}\in \fn{T}_{h}}\Big(\Big(\sum_{i=1}^{n_{p}}e_{k,i}^2\Big)^2\Big)^{1/2}\|\Q_{k}\|_{2}\|\bm{e}_{k}\|_{2}
    \tag{By \cref{lem:sum_of_product}}
    \nonumber
    \\
    & =  {c}_{r}\sum_{\Omega_{k}\in \fn{T}_{h}}\|\bm{e}_{k}\|_{2}^2\|\Q_{k}\|_{2}\|\bm{e}_{k}\|_{2}
    \nonumber
    \\
    & \le  {c}_{r}\|\Q_{*}\|_{2}\sum_{\Omega_{k}\in \fn{T}_{h}}\|\bm{e}_{k}\|_{2}^2 \|\bm{e}_{k}\|_{2}
    \tag{\(\|\Q_{*}\|_{2} =\max_{\Omega_{k}\in\fn{T}_{h}} \|\Q_{k}\|_{2}\)}
    \nonumber
    \\
    & \le {c}_{r}\|\Q_{*}\|_{2}\Big(\sum_{\Omega_{k}\in \fn{T}_{h}}\|\bm{e}_{k}\|_{2}^4\Big)^{1/2} \Big(\sum_{\Omega_{k}\in \fn{T}_{h}}\|\bm{e}_{k}\|_{2}^2\Big)^{1/2}
    \tag{By \cref{lem:sum_of_H_norm_products}}
    \nonumber
    \\
    & \le  {c}_{r}\|\Q_{*}\|_{2}\|\bm{e}\|_{2}^3.\label{eq:eMQe}
\end{align}
Hence, the fourth term in \cref{eq:phi1} adds a similar term as \cref{eq:term2_bound}.

Collecting the bounds in \cref{eq:trace_bound2},  \cref{eq:eMHux}, \cref{eq:eMHtau_u}, \cref{eq:eMQe} and the definition in \cref{eq:c_R}, \(c_{R}=c_r\|\Q_{*}\|_2/2\), we can write the bound for Term III in \cref{eq:energy_error_eq_nonlinear_split0} as 
\begin{align}
    \left|\bm{e}^T\A(\bm{u})\Q\bm{u} - \bm{e}^T\A(\bm{u}_{h})\Q\bm{u}_{h}\right| 
    &\le c_{S}\|\bm e\|_{2}^{2} + 2{c}_{R}\|\bm{e}\|_{2}^3,
    \label{eq:term3_bound}
\end{align}
where \(c_{S}=c_{F}
        + {c}_{r}\,(w_{\max}^{3/2}/w_{\min}^{1/2}) \sqrt{n_p}
        \left(\|\bm{u}_{x,*}\|_{\infty} + \|\bm{\tau}_{u,*}\|_{\infty}\right)\). 

\subsection{Bound for Term IV}
The last term in \cref{eq:energy_error_eq_nonlinear_split0} can be bounded using the triangle and Cauchy-Schwarz inequalities as
\begin{align}
    |\bm{e}^T\H\bm{\tau}| & = \Big|\bm{e}^T\Big(\sum_{\Omega_{k}\in \fn{T}_{h}}\P_{k}^T\H_{k}\P_{k}\Big)\bm{\tau}\Big| = \Big| \sum_{\Omega_{k}\in \fn{T}_{h}}\bm{e}_{k}^T\H_{k}\bm{\tau}_{k}\Big| \nonumber 
    \\
                          & \le \sum_{\Omega_{k}\in \fn{T}_{h}}\Big|\bm{e}_{k}^T\H_{k}\bm{\tau}_{k}\Big|
                         \le \sum_{\Omega_{k}\in \fn{T}_{h}} \|\bm{\tau}_{k}\|_{\H}\|\bm{e}_{k}\|_{\H}.
                         \nonumber
\end{align}
Now, applying \cref{lem:sum_of_H_norm_products} and using the definition of the global errors in \cref{eq:global_error_norms}, we obtain
\begin{equation}\label{eq:term4_bound}
    |\bm{e}^T\H\bm{\tau}| \le \|\bm{\tau}\|_{\H}\|\bm{e}\|_{\H}.
\end{equation}

\subsection{Riccati ODE formulation}
The purpose of this section is to write the error equation in \cref{eq:energy_error_eq_nonlinear_split0} as a constant-coefficient Riccati ODE and determine how its coefficients scale with mesh refinement. To this end, we bound the RHS of \cref{eq:energy_error_eq_nonlinear_split0} using the bounds we have obtained for Term I-IV in \cref{eq:term1_bound,eq:term2_bound,eq:term3_bound,eq:term4_bound}, \ie, we write
\begin{align}
    \frac{1}{2}\der[]{t}\|\bm{e}\|_{\H}^2 &\le
    (\alpha_1 + 2\alpha_2)c_{R} \|\bm{e}\|_{2}^3 + (\alpha_1 c_{F} + \alpha_2 c_{S})\|\bm{e}\|_{2}^2 + \|\bm{\tau}\|_{\H}\|\bm{e}\|_{\H}.
    \label{eq:error_equation_2norm}
\end{align}
We note that 
\begin{equation}
    \|\bm{\tau}\|_{\H}^2 = \bm{\tau}^T\H\bm{\tau}=\sum_{\Omega_{k}\in \fn{T}_{h}}\sum_{i=1}^{n_p}\tau_{k,i}^2 w_{k,i} 
    \le \|\bm{\tau}_*\|_{\infty}^2 \sum_{\Omega_{k}\in \fn{T}_{h}}\sum_{i=1}^{n_p}w_{k,i}
    = \|\bm{\tau}_*\|_{\infty}^2 |\Omega|,
\end{equation}
where \(\|\bm{\tau}_*\|_{\infty}\coloneqq \max_{\Omega_{k}\in \fn{T}_{h}} \|\bm{\tau}_{k}\|_{\infty}\) and \(|\Omega|\) is the volume of the domain. Hence, it follows that 
\begin{equation}\label{eq:tau_H_scaling}
    \|\bm{\tau}\|_{\H} \le \sqrt{|\Omega|} \|\bm{\tau}_*\|_{\infty}.
\end{equation}
Using the norm relations \(\|\bm{e}\|_{2}^2 \le w_{\min}^{-1}\|\bm{e}\|_{\H}^2\) and \cref{eq:tau_H_scaling}, we can write \cref{eq:error_equation_2norm} as
\begin{align}
    \frac{1}{2}\der[]{t}\|\bm{e}\|_{\H}^2 &\le
    (\alpha_1 + 2\alpha_2)c_{R}  w_{\min}^{-3/2}\|\bm{e}\|_{\H}^3 + (\alpha_1 c_{F} + \alpha_2 c_{S})w_{\min}^{-1}\|\bm{e}\|_{\H}^2 + \|\bm{\tau}\|_{\H}\|\bm{e}\|_{\H},
    \nonumber 
    \\
    \der[]{t}\|\bm{e}\|_{\H} &\le
    (\alpha_1 + 2\alpha_2)c_{R} w_{\min}^{-3/2} \|\bm{e}\|_{\H}^2 + (\alpha_1 c_{F} + \alpha_2 c_{S})w_{\min}^{-1}\|\bm{e}\|_{\H} + \sqrt{|\Omega|} \|\bm{\tau}_*\|_{\infty}.
    \label{eq:error_equation_Hnorm}
\end{align}
For \(\|\bm e(t)\|_{\H}>0\), \cref{eq:error_equation_Hnorm} follows by dividing both sides of the preceding inequality by \(\|\bm e(t)\|_{\H}\). When \(\|\bm e(t)\|_{\H}=0\), we have \(\bm e(t)=\bm 0\) since \(\H\) is positive definite. This implies \(\bm{u}(t)=\bm{u}_{h}(t)\) and thus \cref{eq:error_split_form} reduces to
\begin{equation}\label{eq:error_e0}
    \der[\bm{e}(t)]{t} = \bm{\tau}(t).
\end{equation}
The derivative \(\dd{\|\bm{e}\|_{\H}}/\dd{t}\) may not exist when \(\|\bm e(t)\|_{\H}=0\); hence, in such a case, we introduce the upper right Dini derivative as
\begin{equation}\label{eq:dini_def}
    D^+\|\bm{e}(t)\|_{\H}\coloneqq \limsup_{\delta t\to 0^+}\frac{\|\bm{e}(t+\delta t)\|_{\H}-\|\bm{e}(t)\|_{\H}}{\delta t} = \limsup_{\delta t\to 0^+}\frac{\|\bm{e}(t+\delta t)\|_{\H}}{\delta t}.
\end{equation}
The fundamental theorem of calculus and \(\bm e(t)=\bm 0\) give
\begin{equation}\label{eq:e_taylor}
    \bm e(t+\delta t)-\bm e(t)=\bm e(t+\delta t)=\int_t^{t+\delta t}\der[\bm e(s)]{s}\,\dd s.
\end{equation}
Hence, using \cref{eq:dini_def}, \cref{eq:error_e0}, continuity of \(\dd\bm e/\dd t\), and the triangle inequality, we obtain
\begin{align*}
    D^+ \|\bm e(t)\|_{\H}
    &=
    \limsup_{\delta t\to 0^+}\frac{1}{\delta t}\left\|\int_t^{t+\delta t}\der[\bm e]{s}\,\dd s\right\|_{\H}
    \le
    \limsup_{\delta t\to 0^+}\frac{1}{\delta t}\int_t^{t+\delta t}\left\|\der[\bm e]{s}\right\|_{\H}\,\dd s
    \\
    &=
    \left\|\der[\bm e]{t}\right\|_{\H}
    =\|\bm\tau(t)\|_{\H}
    \le \sqrt{|\Omega|}\,\|\bm\tau_*(t)\|_{\infty}.
\end{align*}
Since \(\|\bm{e}(t)\|_{\H}=0\), this is precisely \cref{eq:error_equation_Hnorm} in the upper right Dini derivative sense. Next, recalling that the entries of the \(\H\) matrix are bounded as stated in \cref{eq:bound_on_H_ii}; in particular,
\(
    c h^{d} \le w_{\min} \le C h^{d},
\)
for some constant \(c, C>0\), we determine how the coefficients in \cref{eq:error_equation_Hnorm} scale with the mesh size.
\begin{proposition}\label{prop:constant_coefficient_error_equation}
    The coefficients of \(\|\bm e\|_{\H}^{2}\) and \(\|\bm e\|_{\H}\) in the \(\H\)-norm error-evolution equation, \cref{eq:error_equation_Hnorm}, change with the mesh size as \(h^{-1-d/2}\) and \(h^{0}\) respectively, while \(\|\bm{\tau}_{*}\|_{\infty}\) itself vanishes at a rate of \(h^{p}\).
\end{proposition}
\begin{proof}
    We start with the coefficient of \(\|\bm{e}\|_{2}^{2}\). The expression for \(c_{R}\) is given in \cref{eq:c_R},
    \[
        c_{R}=\frac{c_r}{2}\|\Q_*\|_{2},
    \]
    where we choose \(c_r\) to be the largest constants over the entire space and time duration, \ie, \(c_r=\sup_{s\in\mathbb{R}}|f''(s)|\).
    Now, using the fact that \(\Q_{k} = \fn{O}(h^{d-1})\) (see \cref{lem:sbp_scaling} in \cref{app:sbp_operators}), we can write the coefficient of \(\|\bm{e}\|_{2}^2\) in \cref{eq:error_equation_Hnorm} as
    \begin{align}
        (\alpha_1+ 2\alpha_2) c_R w_{\min}^{-3/2} & = (\alpha_1+ 2\alpha_2)\frac{c_r}{2} \|\Q_*\|_{2}w_{\min}^{-3/2} \nonumber
        = \fn{O}(h^{d-1})\; \fn{O}(h^{-3d/2}) 
        \nonumber
        \\
        \nonumber
                                         & = \fn{O}(h^{-1-d/2}). 
        \label{eq:cR_scaling}
    \end{align}
    
    For the coefficient of \(\|\bm{e}\|_{2}\), we have \(c_{F}\) given in \cref{eq:c_F} as
    \[
        c_{F} =\frac{n_{p}^{3/2}}{2} |\A^{*}|_{\infty}\,\|\Q_*\|_{2}.
    \]
    As a consequence of the smoothness assumptions on the solution and flux functions, in particular taking \(\fn{U}\in C^1\) and \(\fn{F}\in C^2\), and using the compactness of the domain, the flux difference satisfies the Lipschitz continuity condition shown in \cref{app:lipschitz_flux}, \cref{lem:flux_lipschitz},
    \begin{equation}\label{eq:F_diff_mesh}
        |\A^*|_{\infty} \le C h.
    \end{equation}
    Substituting \cref{eq:F_diff_mesh} in the coefficient of \(\|\bm{e}\|_{2}\) in \cref{eq:error_equation_Hnorm}, we find
    \begin{align}
        \frac{\alpha_1}{w_{\min}}c_{F} & = \frac{n_{p}^{3/2}}{2} \frac{\alpha_1}{w_{\min}}|\A^{*}|_{\infty}\,\|\Q_*\|_{2}
        \nonumber
        \\
                                                    & \le \frac{n_{p}^{3/2}}{2}  C \,\frac{\alpha_1}{h^d} h\,\|\Q_*\|_{2}
         =\fn{O}(h^{1-d}) \, \fn{O}(h^{d-1})
        \nonumber
        \\
                                                    & =\fn{O}(h^0).
        \label{eq:cF_scaling}
    \end{align}
    Similarly, using the definition of \(c_S\) in \cref{eq:term3_bound}, we can write
    \begin{equation*}
        \frac{\alpha_2}{w_{\min}}c_S = \frac{\alpha_2}{w_{\min}}c_{F}
        + \alpha_2 c_r\,\left(\frac{w_{\max}}{w_{\min}}\right)^{3/2}
        \sqrt{n_p}
        \left(\|\bm{u}_{x,*}\|_{\infty} + \|\bm{\tau}_{u,*}\|_{\infty}\right).
    \end{equation*}
    Since \(\alpha_1\) and \(\alpha_2\) are constants, the first term on the RHS has a scaling of \(h^{0}\) as shown in \cref{eq:cF_scaling}. We also note that \(c_r\), \(\|\bm{u}_{x,*}\|_{\infty}\), and \(w_{\max}/w_{\min}\) are all mesh independent. Moreover, we have \(\|\bm{\tau}_{u,*}\|_{\infty} =\fn{O}(h^p)\). Hence, it follows that 
    \begin{equation}\label{eq:cS_scaling}
        \frac{\alpha_2}{w_{\min}}c_S =\fn{O}(h^{0}).
    \end{equation}
    
    Finally, since \(\tau_{k,i}=\fn{O}(h^{p})\) for \(i\in\{1,\dots,n_p\}\) and \(\Omega_{k}\in \fn{T}_{h}\), it follows that \(\|\bm{\tau}_{*}\|_{\infty} \) vanishes with the mesh size at a rate of \(h^p\), as stated in a vector form in \cref{eq:tau_csbp_split_form}. 
\end{proof}

\subsection{Convergence analysis}

In order to use the error-evolution inequality, \cref{eq:error_equation_Hnorm}, for convergence analysis, we first find an associated upper envelope ODE that we can solve explicitly.  
\begin{proposition}\label{prop:riccati_envelope}
    Since the coefficient functions \(c_R(t)\), \(c_F(t)\), \(c_S(t)\), and the truncation error, \(\|\bm\tau_*(t)\|_{\infty}\), are continuous on \([0,T]\), we define for a fixed mesh size \(h\) satisfying \(0 < h \le h_{1}\)
    \begin{align}
        a &\coloneqq a(h) = \sup_{0\le t\le T} (\alpha_1 + 2\alpha_2)c_{R}(t,h) w_{\min}^{-3/2}(h),
        \\
        b& \coloneqq\sup_{0 < \tilde{h} \le h_{1}}\sup_{0\le t\le T} w_{\min}^{-1}(\tilde{h})(\alpha_1 c_{F}(t,\tilde{h}) + \alpha_2 c_{S}(t,\tilde{h})),
        \\
        c &\coloneqq c(h) =  \sup_{0\le t\le T} \sqrt{|\Omega|} \|\bm\tau_*(t,h)\|_{\infty},
    \end{align}
    which are finite, provided \(h_{1}\) is sufficiently small. Using the fact that the numerical solution is bounded in the sense of \cref{eq:solution_bound}, there exists \(M<\infty\) such that
    \[
        z(t)\coloneqq\|\bm e(t)\|_{\H} \le M \qquad \forall t\in[0,T].
    \]
    Then, due to \cref{eq:error_equation_Hnorm} and because \(w_{\min}^{-1}(h)(\alpha_1 c_{F}(t,h) + \alpha_2 c_{S}(t,h))\le b\), \(z(t)\) satisfies the differential inequality
    \begin{equation}\label{eq:z_inequality}
        \frac{\mathrm{d}z}{\mathrm{d}t} \le a\,z^2 + b\,z + c, \qquad \forall t\in[0,T].
    \end{equation}
    Let \(y(t)\) be the maximal solution of the constant-coefficient Riccati initial-value problem
    \begin{equation}\label{eq:riccati_energy_error}
        \der[y]{t}= a\,y^2 + b\,y + c,\qquad y(0)=z(0)=0.
    \end{equation}
    Then, \(z(t)\le y(t)\) for all \(t\) in the common interval of existence. In particular, analyzing the Riccati ODE for \(y(t)\) provides an upper envelope for the error norm.
\end{proposition}

\begin{proof}
    Let \(G(t,s):=a s^2 + b s + c\). Since \(G\) is continuous in \(t\) and locally Lipschitz in \(s\), \ie, for any \(R>0\), all \(t\in[0,T]\), and all \(s,w\in[-R,R]\),
    \[
        |G(t,s)-G(t,w)| = |a(s^2-w^2)+b(s-w)| \le |s-w|\,\bigl(|a||s+w|+|b|\bigr)
        \le |s-w|\,(2|a|R+|b|),
    \]
    by Picard--Lindel{\"o}f theorem, there exists a unique solution near \(t=0\) with \(y(0)=z(0)\). Note that the stability hypothesis, \ie, \(\|\bm{e}\|_{2}<\infty\) or equivalently \(\|\bm{u}_{h}\|_{\H}<\infty\), implies that \(z\) is bounded on \([0,T]\) and thus cannot blow up over this time interval. In contrast, the Riccati solution, \(y(t)\), may blow up in finite time when \(a>0\). 

    Let \(I\) denote the common interval of existence of \(z\) and \(y\). Since \(z(t)=\|\bm e(t)\|_{\H}\) is continuous, and since \cref{eq:z_inequality} holds in the classical sense when \(z(t)>0\) and in the upper right Dini derivative sense when \(z(t)=0\), we have
    \[
        D^+z(t)\le G(t,z(t)),\qquad z(0)=y(0)=0.
    \]
    Therefore, the standard comparison lemma for scalar differential inequalities, see \eg, \cite[Lemma 3.4]{khalil2002nonlinear}, implies that \(z(t)\le y(t)\) for all \(t\in I\).
\end{proof}

\begin{remark}
    On a finite mesh, the coefficients in \cref{eq:riccati_energy_error}, play different roles in terms of impacting the error growth and blow-up times. The coefficients dictate the maximum error growth that may be achieved by a specific discretization. We summarize the impact of the coefficients on the behaviour of the solution of \cref{eq:riccati_energy_error} as follows:
    \begin{itemize}
        \item \(a\) is associated with the magnitude of the Hessian of the flux with respect to the solution, \ie, \(f''(\bm{u})\). Hence, its value is zero for linear PDEs.
        \item \(b\) is associated with the maximum entry of the flux Jacobian difference matrix, \(|\A^*|_{\infty}\), which measures how fast the flux Jacobian changes on a given stencil, \ie, \(f_{i}'-f_{j}'\).
        \item \(c\) is associated with the truncation error, and hence, can be reduced by optimizing operators to reduce the truncation error coefficient.
        \item For a fixed \(b\), \(a\) and \(c\) govern the error growth and blow-up mode by controlling the sign of \(b^2 - 4ac\).
    \end{itemize}
\end{remark}

The main idea for the convergence analysis is to bound \(z(t)\) by \(y(t)\) and show that \(y(t)\) tends to zero with mesh refinement. However, as stated earlier, \(y(t)\) can blow up in finite time; this is established in the proposition below.
\begin{proposition}\label{prop:csbp_error_estimate_divergence_form}
    On a finite mesh, the solution to the Riccati ODE \cref{eq:riccati_energy_error}, associated with the C-SBP discretization given in \cref{eq:csbp_disc_split_form}, of the scalar hyperbolic PDE \cref{eq:nonlinear_prob} can blow up in finite time unless \(a = 0\) or \(c=0\). Furthermore, the blow-up time is completely determined by the coefficients of the error equation. In the specific case when \(a=0\) and \(b\neq 0\), the solution can grow exponentially in time, and if \(a=b=0\), then it can grow only up to linearly in time.
\end{proposition}
\begin{proof}
    The proof is included in the appendix, see \cref{sec:riccati_blow_up_time}.
\end{proof}
\begin{remark}
    The blow-up times stated in \cref{prop:csbp_error_estimate_divergence_form} are specific to the Riccati equation \cref{eq:riccati_energy_error}, not the actual error equation \cref{eq:error_equation_Hnorm}; the latter remains bounded, as the numerical error is bounded due to the stability of the semi-discrete scheme.
\end{remark}

To be able to use the idea of bounding \(z(t)\) by \(y(t)\) for convergence purposes, we must show that there exists a mesh size, \(h>0\), for which \(y(t)\) is bounded for the time interval \([0,T]\) as \(h\to 0\). The next lemma shows this is the case.

\begin{lemma}\label{lem:riccati_no_blowup_smallh}
Assume that the numerical solution is stable in the sense of \cref{eq:energy_2_disc} and that, by \cref{prop:constant_coefficient_error_equation,prop:riccati_envelope}, the coefficients of the constant-coefficient Riccati initial-value problem \cref{eq:riccati_energy_error} satisfy 
\[
    a =\fn{O}(h^{-1-d/2}), \qquad b \ge 0 \text{ (constant)}, \qquad c =\fn{O}(h^{p}),
\]
with \(p>1+d/2\), and that \(a\ge 0\), \(b\ge0\), and \(c\ge 0\) for all sufficiently small \(h\). Then, for any fixed final time \(T>0\), there exists \(h_0=h_0(T)>0\) such that the solution \(y(t)\) of \cref{eq:riccati_energy_error} exists on the entire interval \([0,T]\) for all \(0<h<h_0\). Equivalently, the blow-up time \(t_*(h)\) of the Riccati solution satisfies
\[
    t_*(h) > T, \qquad \forall\, 0<h<h_0.
\]
\end{lemma}

\begin{proof}
    Fix \(T>0\). By construction, for all sufficiently small \(h\) one has \(a\ge 0\), \(b\ge 0\), \(c\ge 0\), with \(a=\fn{O}(h^{-1-d/2})\), \(c =\fn{O}(h^{p})\) and \(p>1+d/2\). In particular, 
    \[
        ac=\fn{O}(h^{p-1-d/2})\to 0 \qquad \text{as } h\to 0.
    \]
    If \(a=0\), then \cref{eq:riccati_energy_error} is linear and the solution exists for all time, so \(t_*(h)=\infty\). Hence, we assume \(a>0\). We also note that when \(a>0\), \(b\ge 0\), and \(c=0\), the Riccati problem becomes, 
    \[
        y'(t) = ay(t)^2 + by(t), \qquad y(0) = 0.
    \]
    Since \(y(t)=0\) satisfies the above, by uniqueness of the solution, we conclude that the solution exists at all times. This means \(t_*(h)=\infty\) for this case too; hence, we further assume \(c>0\), and consider the following two cases only.

    \paragraph{Case 1. (\(a>0\), \(b>0\), and \(c>0\)).}
    Since \(b>0\) is constant, \(ac\to 0\) as \(h\to 0\) implies the existence of \(h_{0}>0\) such that \(4ac < b^{2}\) for all \(0<h<h_{0}\), and hence \(\Delta(h)>0\) on \(0<h<h_{0}\), where \(\Delta (h)\coloneqq b^2-4ac\). In this regime, the blow-up time is given in \cref{eq:case4_blow_up} as
    \[
        t_*(h)
        = \frac{1}{\sqrt{\Delta}}
        \ln\!\left(\frac{-b-\sqrt{\Delta}}{-b+\sqrt{\Delta}}\right).
    \]
    Moreover, we can write
    \(
        (-b-\sqrt{\Delta})(-b+\sqrt{\Delta})=b^2-\Delta=4ac
    \)
    and thus, we obtain
    \[
        \frac{-b-\sqrt{\Delta}}{-b+\sqrt{\Delta}}
        =\frac{(-b-\sqrt{\Delta})^2}{(-b-\sqrt{\Delta})(-b+\sqrt{\Delta})}
        =\frac{(b+\sqrt{\Delta})^2}{4ac}
        \ge \frac{b^2}{4ac}.
    \]
    Since \(\sqrt{\Delta}\le b\), and the logarithm is an increasing function, it follows
    \[
        t_*(h)
        = \frac{1}{\sqrt{\Delta}}\ln\!\left(\frac{(b+\sqrt{\Delta})^2}{4ac}\right)
        \ge \frac{1}{b}\ln\!\Big(\frac{b^2}{4ac}\Big).
    \]
    Because \(ac\to 0\), the right-hand side diverges to \(+\infty\) as \(h\to 0\), so there exists \(h_0=h_0(T)>0\) such that \(t_*(h)>T\) for all \(0<h<h_0\).

    \paragraph{Case 2. (\(a>0\), \(b=0\), and \(c>0\)).}
    In this case, we have \(\Delta(h)=b^2-4ac=-4ac<0\) for all sufficiently small \(h\) with \(c>0\), and the blow-up time is given in \cref{eq:case3_blowup} as
    \[
        t_*(h)=\frac{\pi}{2\sqrt{ac}}.
    \]
    Since \(ac=\fn{O}(h^{p-1-d/2})\to 0\) with \(p>1+d/2\), we have \(t_*(h)\to\infty\) as \(h\to 0\). Therefore, for any fixed \(T>0\) there exists \(h_0=h_0(T)>0\) such that \(t_*(h)>T\) for all \(0<h<h_0\). Hence, in both cases, the solution exists on \([0,T]\) for all sufficiently small \(h\), as desired.
\end{proof}

Now that we can bound \(z(t)\) by \(y(t)\) at all time below some mesh resolution by \cref{lem:riccati_no_blowup_smallh}, it only remains to show that \(y(t)\to 0\) as \(h\to 0\) to establish convergence.

\begin{theorem}\label{thm:csbp_convergence}
Let \(\fn{U}\in C^1(\Omega\times[0,T])\) be a solution of \cref{eq:nonlinear_prob} for some fixed \(T>0\), and assume the scalar flux \(\fn{F}:\mathbb{R}\to\mathbb{R}\) satisfies \(\fn{F}\in C^2(\mathbb{R})\) with globally bounded second derivative. Suppose the stable split-form C-SBP discretization given in \cref{eq:csbp_disc_split_form} is consistent of order \(p>1+d/2\) in the sense that the truncation error satisfies \(\|\bm\tau(t)\|_\infty=\fn{O}(h^p)\) uniformly for \(t\in[0,T]\). Assume further that \(\bm e(0)=\bm 0\). Then, for each \(t\in[0,T]\), the numerical solution converges, \ie,
\[
    \lim_{h\to 0}\|\bm u(t)-\bm u_h(t)\|_{\H} = 0.
\]
\end{theorem}
\begin{proof}
Fix any \(t\in[0,T]\). By \cref{prop:riccati_envelope}, we have
\[
    \|\bm e(s)\|_{\H} = z(s)\le y(s), \qquad \forall s\in[0,T],
\]
whenever the Riccati envelope, \(y\), exists on \([0,T]\). By \cref{lem:riccati_no_blowup_smallh}, there exists \(h_0=h_0(T)>0\) such that for all \(0<h<h_0\) the solution \(y(t)\) of \cref{eq:riccati_energy_error} exists on \([0,T]\). Hence, it suffices to show
\[
    \lim_{h\to 0} y(t)=0,
\]
and then conclude \(\lim_{h\to 0}\|\bm u(t)-\bm u_h(t)\|_{\H}=0\) by domination.

By \cref{prop:constant_coefficient_error_equation,prop:riccati_envelope}, we have
\[
    a=\fn{O}(h^{-1-d/2}),\qquad b \ge 0 \text{ (constant)},\qquad c=\fn{O}(h^{p}),
\]
hence, as \(h\to 0\), we have \(c(h)\to 0\) and, for \(p>1+d/2\), \(ac =\fn{O}(h^{p-1-d/2}) \to 0\). This implies that for sufficiently small \(h\) and \(a,c>0\), either \(b>0\) and then \(\Delta(h)>0\), or \(b=0\) and then \(\Delta(h) < 0\). Note that the case \(c=0\) is trivial, since the solution becomes \(y(t)= 0\) and convergence is immediate. Therefore, we consider the following four coefficient regimes of the constant-coefficient Riccati ODE, whose solutions are given in \cref{app:riccati}.

\paragraph{Case 1.} \(a=0\) and \(b=0\).
Then \(y'=c\), \(y(0)=0\), so \(y(t)=ct\). Since \(c(h)\to 0\), we have \(y(t)\to 0\).

\paragraph{Case 2.} \(a=0\) and \(b > 0\). Then, \(y'=by+c\), \(y(0)=0\), and the solution is given by \cref{eq:case2_solution}
\[
    y(t)=\frac{c}{b}\left(e^{bt}-1\right).
\]
Here, \(c(h)\to 0\). Moreover, \(b=\fn{O}(h^{0})\), so \(e^{bt}-1\) stays bounded along any sequence \(h\to 0\) for which this case holds, and thus \(y(t)\to 0\).

\paragraph{Case 3.} \(a>0\), \(b=0\), and \(c>0\).
Then, \(y'=ay^2+c\), \(y(0)=0\), and we obtain the solution given in  \cref{eq:case3_solution},
\[
    y(t)=\sqrt{\frac{c}{a}}\tan\bigl(\sqrt{ac}\,t\bigr).
\]
Since \(ac\to 0\), we have \(\sqrt{ac}\,t\to 0\), and using \(\tan x \sim x\) as \(x\to 0\),
\[
    y(t)\sim \sqrt{\frac{c}{a}}\; \sqrt{ac}\,t = ct \to 0.
\]

\paragraph{Case 4.} \(a> 0\), \(b> 0\), and \(\Delta=b^2-4ac>0\).
The explicit solution is \cref{eq:solution_Dpos},
\[
    y(t)=\frac{r_1\bigl(1-e^{\sqrt{\Delta}\,t}\bigr)}{1-(r_1/r_2)e^{\sqrt{\Delta}\,t}},
    \qquad
    r_1=\frac{-b+\sqrt{\Delta}}{2a},\quad r_2=\frac{-b-\sqrt{\Delta}}{2a}.
\]
First, note that since \((-b+\sqrt{\Delta})(-b-\sqrt{\Delta})=4ac\), we have
\[
    r_1=-\frac{2c}{b+\sqrt{\Delta}}.
\]
Hence, \(r_1\to 0\) because \(c(h)\to 0\) and \(b+\sqrt{\Delta}\ge b>0\) in this case. Next, we note that \(\sqrt{\Delta}\) is \(\fn{O}(1)\), since \(b >0\), which implies \(e^{\sqrt{\Delta}t}\) remains bounded as \(h\to 0\). Finally, \(r_1/r_2=(b-\sqrt{\Delta})/(b+\sqrt{\Delta})=4ac/(b+\sqrt{\Delta})^2 \to 0\) as \(h\to 0\), so the denominator, \(1-(r_1/r_2)e^{\sqrt{\Delta}t}\to 1\) for all sufficiently small \(h\) (satisfying \(t\le T<t_*(h)\) in \cref{lem:riccati_no_blowup_smallh}). Therefore, the numerator tends to \(0\) (since \(r_1\to 0\) and \(e^{\sqrt{\Delta}t}\) is bounded),
and we conclude \(y(t)\to 0\). Since \(t\) is fixed arbitrarily, the result holds for any \(t\in[0,T]\).
\end{proof}

\cref{thm:csbp_convergence} establishes convergence for the split-form discretization of the scalar hyperbolic PDE \cref{eq:nonlinear_prob} under suitable assumptions; however, it also indicates that other forms of stable discretizations, including those for systems of equations, converge under similar assumptions if their error equations can be written in the form of the constant-coefficient Riccati equation \cref{eq:riccati_energy_error}, with coefficients that satisfy the mesh-dependence conditions in \cref{prop:constant_coefficient_error_equation}. Our task in the next section is to show that these conditions are satisfied for the split-form discretizations of symmetric systems of hyperbolic PDEs.

\section{Convergence analysis for nonlinear symmetric systems of equations}\label{sec:systems}
We now consider the entropy-stable split-form discretization introduced in \cref{sec:split_form_stability} for systems of equations with homogeneous fluxes. As in the previous section, for notational clarity, we present the analysis for a single direction and suppress the spatial index \(i\); the multidimensional result follows by summing identical contributions over \(i\in\{1,\dots,d\}\). Furthermore, we retain the scalar notation for the coefficient symbols \(c_R\), \(c_F\), \(c_S\), \(c_r\), etc., with the understanding that in the systems case they denote the corresponding systems-level constants obtained from block-vector estimates.

\paragraph{Semi-discrete split form.}
The (single-direction) split-form C-SBP semi-discretization of \cref{eq:nonlinear_prob} is
\begin{equation}\label{eq:systems_split_form_disc_1d}
    \der[\bm{u}_h]{t}
    = -\alpha_1\,\overline{\D}\,\bm{f}(\bm{u}_h) - \alpha_2\,\bar{\A}(\bm{u}_h)\,\overline{\D}\,\bm{u}_h,
\end{equation}
where \(\alpha_1\) and \(\alpha_2\) are the constant splitting coefficients that coincide with those in \cref{eq:split_form_disc}. The matrix \(\bar{\A}\left(\bm{u}_h\right)\) is block diagonal with \(\A\left(\bm{u}_{h,k}^{(m)}\right)\in\IR{n_c\times n_c}\) at diagonal block \((k-1)n_p+m\), and \(\overline{\D}=\bigl[\overline{\H}\bigr]^{-1}\overline{\Q}\) is the global C-SBP differentiation operator constructed as in \cref{eq:global_D_construction} using Kronecker products with \(\I_{n_c}\).

\paragraph{Truncation error and error equation.}
Substituting the exact solution into \cref{eq:systems_split_form_disc_1d} defines the truncation error vector
\begin{equation}\label{eq:systems_split_form_trunc}
    \bm{\tau}
    \coloneqq
    \der[\bm{u}]{t}
    +\alpha_1\,\overline{\D}\,\bm{f}(\bm{u})
    +\alpha_2\,\bar{\A}(\bm{u})\,\overline{\D}\,\bm{u}
    = \fn{O}(h^p),
\end{equation}
where \(p\) is the consistency order of the C-SBP operator. Subtracting \cref{eq:systems_split_form_disc_1d} from \cref{eq:systems_split_form_trunc} and using \(\bm{e}\coloneqq \bm{u}-\bm{u}_h\), we obtain
\begin{equation}\label{eq:systems_split_form_error_eq0}
    \der[\bm{e}]{t}
    =
    -\alpha_1\,\overline{\D}\bigl(\bm{f}(\bm{u})-\bm{f}(\bm{u}_h)\bigr)
    -\alpha_2\Bigl(\bar{\A}(\bm{u})\,\overline{\D}\bm{u}-\bar{\A}(\bm{u}_h)\,\overline{\D}\bm{u}_h\Bigr)
    +\bm{\tau}.
\end{equation}

\paragraph{Taylor expansion of the flux.}
Applying a Taylor expansion (componentwise on each node) gives
\begin{equation}\label{eq:systems_flux_taylor}
    \bm{f}(\bm{u})-\bm{f}(\bm{u}_h)
    =
    \bar{\A}(\bm{u})\,\bm{e} - \bm{R}_f(\bm{e}),
\end{equation}
where \(\bm{R}_f(\bm{e})\) is a remainder satisfying the pointwise quadratic bound
\begin{equation}\label{eq:systems_flux_remainder_bound}
    \|\bm{R}_{f,k}^{(m)}(\bm{e})\|_2 \le \frac{c_r}{2}\left\|\bm{e}_{k}^{(m)}\right\|_2^2,
\end{equation}
with a mesh-independent constant \(c_r\) determined by the supremum of the Hessian of the component fluxes with respect to \(\fnb{U}\).
Substituting \cref{eq:systems_flux_taylor} into \cref{eq:systems_split_form_error_eq0} yields
\begin{align}
    \der[\bm{e}]{t}
    &=
    -\alpha_1\,\overline{\D}\,\bar{\A}(\bm{u})\,\bm{e}
    +\alpha_1\,\overline{\D}\,\bm{R}_f(\bm{e})
    -\alpha_2\Bigl(\bar{\A}(\bm{u})\,\overline{\D}\bm{u}-\bar{\A}(\bm{u}_h)\,\overline{\D}\bm{u}_h\Bigr)
    +\bm{\tau}.
    \label{eq:systems_split_form_error_eq1}
\end{align}

\paragraph{Energy identity.}
Premultiplying \cref{eq:systems_split_form_error_eq1} by \(\bm{e}^T\overline{\H}\) and simplifying gives
\begin{align}\label{eq:systems_split_form_error_energy0}
    \frac{1}{2}\der[]{t}\|\bm{e}\|_{\H}^2
    &=
    -\alpha_1\,\bm{e}^T\overline{\Q}\,\bar{\A}(\bm{u})\,\bm{e}
    +\alpha_1\,\bm{e}^T\overline{\Q}\,\bm{R}_f(\bm{e})
    -\alpha_2\,\bm{e}^T\overline{\H}\Bigl(\bar{\A}(\bm{u})\,\overline{\D}\bm{u}-\bar{\A}(\bm{u}_h)\,\overline{\D}\bm{u}_h\Bigr)
    +\bm{e}^T\overline{\H}\bm{\tau}.
\end{align}
The last term in \cref{eq:systems_split_form_error_energy0} is bounded exactly as in the scalar case, \cref{eq:term4_bound},
\begin{equation}\label{eq:systems_split_form_trunc_bound}
    |\bm{e}^T\overline{\H}\bm{\tau}|\le \|\bm{\tau}\|_{\H}\,\|\bm{e}\|_{\H}.
\end{equation}
Similarly, the second term in \cref{eq:systems_split_form_error_energy0} is bounded as in \cref{eq:term2_bound},
\begin{equation}\label{eq:systems_split_form_eQR}
    \bigl|\bm{e}^T\overline{\Q}\,\bm{R}_f(\bm{e})\bigr|
    \le
    \frac{c_r}{2}\|\overline{\Q}_*\|_2\,\|\bm{e}\|_2^3
    = c_R\,\|\bm{e}\|_2^3,
\end{equation}
with \(\|\overline{\Q}_*\|_{2}=\max_{\Omega_k}\|\overline{\Q}_k\|_2\) and \(c_R = \frac{c_r}{2}\|\overline{\Q}_*\|_2\).

Using a similar derivation as \cref{eq:hadamard_QF} for the scalar case, it can be shown that the first term on the RHS of \cref{eq:systems_split_form_error_energy0} has the structure
\begin{equation}\label{eq:systems_split_form_eQAe_def}
    \bm{e}^T\overline{\Q}\,\bar{\A}(\bm{u})\,\bm{e}
    =
    \frac{1}{2}\sum_{\Omega_k\in\fn{T}_h}
    \bm{e}_k^T\Bigl(\widehat{\Q}_k\circ \A^*(\bm{u}_k)\Bigr)\bm{e}_k,
\end{equation}
where \(\widehat{\Q}_k\coloneqq \Q_k\otimes\mathbbm{1}\) with \(\mathbbm{1}\in\IR{n_c\times n_c}\) containing ones, and \(\A^*(\bm{u}_k)\) is the block matrix with \((i,j)\)-block entry
\[
    [\A^*(\bm{u}_k)]_{ij} = \A(\bm{u}_k^{(j)})-\A(\bm{u}_k^{(i)}).
\]
Smoothness of \(\A(\cdot)\) implies \(\|\A^*(\bm{u}_k)\|_\infty=\fn{O}(h)\); hence, we obtain
\begin{equation}\label{eq:systems_split_form_eQAe_bound}
    \bigl|\bm{e}^T\overline{\Q}\,\bar{\A}(\bm{u})\,\bm{e}\bigr|
    \le c_F\,\|\bm{e}\|_2^2,
\end{equation}
where \(c_F\coloneqq \frac{\sqrt{n_p\,n_c}}{2}\|\A^*_*\|_\infty\,\|\widehat{\Q}_*\|_2\).

Finally, the remaining term in \cref{eq:systems_split_form_error_energy0} is the \(\alpha_2\) split-form contribution
\begin{equation}\label{eq:systems_split_form_phi0}
    \phi
    \coloneqq
    \alpha_2\,\bm{e}^T\overline{\H}\bigl(\bar{\A}(\bm{u})\,\overline{\D}\bm{u}-\bar{\A}(\bm{u}_h)\,\overline{\D}\bm{u}_h\bigr).
\end{equation}
Using \(\bm{u}=\bm{u}_h+\bm{e}\) and \(\overline{\H}\overline{\D}=\overline{\Q}\), we rewrite
\begin{align}
    \phi
    &=
    \alpha_2\,\bm{e}^T\bar{\A}(\bm{u})\,\overline{\Q}\bm{u}
    -\alpha_2\,\bm{e}^T\bar{\A}(\bm{u}_h)\,\overline{\Q}\bm{u}_h
    \nonumber
    \\
    &=
    \alpha_2\,\bm{e}^T\bar{\A}(\bm{u})\,\overline{\Q}\bm{e}
    +\alpha_2\,\bm{e}^T\bigl(\bar{\A}(\bm{u})-\bar{\A}(\bm{u}_h)\bigr)\overline{\Q}\bm{u}
    -\alpha_2\,\bm{e}^T\bigl(\bar{\A}(\bm{u})-\bar{\A}(\bm{u}_h)\bigr)\overline{\Q}\bm{e}.
    \label{eq:systems_split_form_phi1}
\end{align}
We bound the three terms on the RHS of \cref{eq:systems_split_form_phi1} in turn.

\paragraph{Term 1.}
The first term in \cref{eq:systems_split_form_phi1} has the same Hadamard structure as \cref{eq:systems_split_form_eQAe_def} and is bounded similarly:
\begin{equation}\label{eq:systems_split_form_phi_term1}
    \bigl|\bm{e}^T\bar{\A}(\bm{u})\,\overline{\Q}\bm{e}\bigr|
    \le c_F\,\|\bm{e}\|_2^2.
\end{equation}

For the remaining two terms on the RHS of \cref{eq:systems_split_form_phi1}, we consider the Taylor expansion of \(\A(\cdot)\) about \(\bm{u}\). Since \(\A\) is the Jacobian of \(\fnb{F}\), smoothness of \(\fnb{F}\) implies \(\A\in C^1\), and a Taylor expansion at each node gives
\begin{equation}\label{eq:systems_A_expand}
    \bar{\A}(\bm{u}_h)=\bar{\A}(\bm{u})-\bar{\R}_{A}(\bm{e}),
\end{equation}
where \(\bar{\R}_{A}(\bm{e})\) is block diagonal matrix that satisfies a pointwise bound of the form
\begin{equation}\label{eq:systems_RA_bound}
    \|\bar{\R}_{A,k}^{(m)}(\bm{e})\|_2 \le {c}_r\,\|\bm{e}_{k}^{(m)}\|_2.
\end{equation}

\paragraph{Term 2.}
Using \cref{eq:systems_A_expand}, we bound the second term in \cref{eq:systems_split_form_phi1} by Cauchy--Schwarz:
\begin{align}
    \big|\bm{e}^T\bigl(\bar{\A}(\bm{u})-\bar{\A}(\bm{u}_h)\bigr)\overline{\Q}\bm{u}\big|
    = \bigl|\bm{e}^T\bar{\R}_{A}(\bm{e})\,\overline{\Q}\bm{u}\bigr|
    &=
    \bigl|\bm{e}^T\bar{\R}_{A}(\bm{e})\,\overline{\H}\overline{\D}\bm{u}\bigr|
    \nonumber
    \\
    &\le
    \|\bm{e}^T\bar{\R}_{A}(\bm{e})\|_{\H}\,\|\overline{\D}\bm{u}\|_{\H}.
    \label{eq:systems_split_form_phi_term2_0}
\end{align}
Proceeding as in the scalar split-form estimate \cref{eq:eMH} (since \(\bar{\R}_{A}(\bm{e})\) is block diagonal and satisfies \cref{eq:systems_RA_bound}), we obtain
\begin{equation}\label{eq:systems_eRAH_bound}
    \|\bm{e}^T\bar{\R}_{A}(\bm{e})\|_{\H}
    \le {c}_r\,\frac{1}{\sqrt{w_{\min}}}\,\|\bm{e}\|_{\H}^2.
\end{equation}
Substituting \cref{eq:systems_eRAH_bound} into \cref{eq:systems_split_form_phi_term2_0}, using \(\|\bm{v}\|_{\H}\le \sqrt{w_{\max}}\sqrt{n_p n_c}\|\bm{v}\|_{\infty}\), noting that \(\|\overline{\D}\bm{u}\|_{\H} \le \|\bm{u}_{x}\|_{\H} + \|\bm{\tau}_{u}\|_{\H}\), and repeating similar steps used for \cref{eq:term23_1}, we find 
\begin{align}
    \bigl|\bm{e}^T\bar{\R}_{A}(\bm{e})\,\overline{\Q}\bm{u}\bigr|
    &\le
    {c}_r\, w_{\max}\frac{\sqrt{w_{\max}}}{\sqrt{w_{\min}}}\sqrt{n_p n_c}\,(\|\bm{u}_{x,*}\|_{\infty} + \|\bm{\tau}_{u,*}\|_{\infty})\,\|\bm{e}\|_{2}^2.
    \label{eq:systems_split_form_phi_term2}
\end{align}
Since \(\bm{u}\) is smooth, \(\|\bm{u}_{x,*}\|_{\infty}\) is bounded, and \(\|\bm{\tau}_{u,*}\|_{\infty}=\fn{O}(h^{p})\) vanishes as \(h\to0\); hence, \cref{eq:systems_split_form_phi_term2} contributes a quadratic term of the same type as \cref{eq:systems_split_form_eQAe_bound}. In particular, it modifies the coefficient multiplying \(\|\bm{e}\|_{2}^2\).

\paragraph{Term 3.}
Finally, the third term in \cref{eq:systems_split_form_phi1} is cubic in \(\bm{e}\). Using \cref{eq:systems_A_expand} and repeating the steps used to obtain \cref{eq:eMQe} in the scalar case, we find
\begin{align}
    \bigl|\bm{e}^T\bar{\R}_{A}(\bm{e})\,\overline{\Q}\bm{e}\bigr|
    &\le
    \sum_{\Omega_k\in\fn{T}_h}
    \|\bm{e}_k^T\bar{\R}_{A}(\bm{e}_k)\|_2\,\|\overline{\Q}_k\bm{e}_k\|_2
    \nonumber
    \le
    {c}_r\,\|\overline{\Q}_*\|_2\sum_{\Omega_k\in\fn{T}_h}\|\bm{e}_k\|_2^3
    \nonumber
    \\
    &\le
    {c}_r\,\|\overline{\Q}_*\|_2\,\|\bm{e}\|_2^3
    \nonumber
    \\
    &=
    2{c}_R \|\bm{e}\|_2^3.
    \label{eq:systems_split_form_phi_term3}
\end{align}
Thus, \cref{eq:systems_split_form_phi_term3} modifies the cubic coefficient multiplying \(\|\bm{e}\|_2^3\). As in the scalar case, we define \(c_S\) as 
\begin{equation}
    c_{S} = c_F + {c}_r\,w_{\max}\frac{\sqrt{w_{\max}}}{\sqrt{w_{\min}}}\sqrt{n_p n_c}\,(\|\bm{u}_{x,*}\|_{\infty} + \|\bm{\tau}_{u,*}\|_{\infty}). 
\end{equation}

\paragraph{Riccati-type error inequality.}
Substituting the bounds \cref{eq:systems_split_form_trunc_bound}, \cref{eq:systems_split_form_eQR}, \cref{eq:systems_split_form_eQAe_bound}, \cref{eq:systems_split_form_phi_term1}, \cref{eq:systems_split_form_phi_term2}, and \cref{eq:systems_split_form_phi_term3} into \cref{eq:systems_split_form_error_energy0}, changing the 2-norm error terms into \(\H\)-norms, and applying \cref{eq:tau_H_scaling}, we obtain an inequality of the same form as \cref{eq:error_equation_Hnorm} obtained for the scalar case, \ie, 
\begin{equation*}\label{eq:systems_split_form_energy_ineq_H}
    \der[]{t}\|\bm{e}\|_{\H}
    \le
    (\alpha_1 + 2\alpha_2)c_{R}w_{\min}^{-3/2} \|\bm{e}\|_{\H}^2 + w_{\min}^{-1}(\alpha_1 c_{F} + \alpha_2 c_{S})\|\bm{e}\|_{\H} + \sqrt{|\Omega|}\|\bm{\tau}_{*}\|_{\infty},
\end{equation*}
which is in the Riccati ODE form as desired,
\begin{equation}\label{eq:systems_split_form_error_ineq_2}
    \der[]{t}\|\bm{e}\|_{\H}
    \le
    a\,\|\bm{e}\|_{\H}^{2}
    +b\,\|\bm{e}\|_{\H}
    +c,
\end{equation}
where, \(a\), \(b\), and \(c\) are defined as in \cref{prop:riccati_envelope} and satisfy the mesh scaling conditions in \cref{prop:constant_coefficient_error_equation}. 

We now state the convergence result for the split-form C-SBP discretization of the symmetric hyperbolic system \cref{eq:nonlinear_prob_split} with homogeneous fluxes. We note that entropy stability of the split-form discretization for systems was established in \cref{sec:split_form_stability}, which guarantees that \(\|\bm{u}_h(t)\|_{\H}\) remains bounded on \([0,T]\).

\begin{theorem}\label{thm:csbp_convergence_systems_split_form}
Consider the symmetric hyperbolic system \cref{eq:nonlinear_prob} on a periodic domain and assume:
\begin{itemize}
    \item the exact solution satisfies \(\fnb{U}\in C^1(\Omega\times[0,T])\),
    \item each flux \(\fnb{F}_i:\mathbb{R}^{n_c}\to\mathbb{R}^{n_c}\) belongs to \(C^2(\mathbb{R}^{n_c})\), is homogeneous of degree \(\beta\), so that \cref{eq:homogeneity} holds, and has globally bounded second derivatives,
    \item the C-SBP discretization \cref{eq:systems_split_form_disc_1d} (summed over \(i=1,\dots,d\)) is consistent of order \(p>1+d/2\) in the sense that \(\|\bm{\tau}_{*}(t)\|_\infty=\fn{O}(h^p)\) uniformly for \(t\in[0,T]\),
    \item the error at \(t=0\) satisfies \(\bm{e}(0)=\bm{0}\).
\end{itemize}
Then, for each fixed \(t\in[0,T]\), the numerical solution produced by the split-form C-SBP discretization, \cref{eq:split_form_disc}, converges under mesh refinement, \ie,
\[
    \lim_{h\to 0}\|\bm{u}(t)-\bm{u}_h(t)\|_{\H}=0.
\]
\end{theorem}

\begin{proof}
The proof is the same as that of \cref{thm:csbp_convergence}. 
\end{proof}

\begin{remark}
Different discretizations and discrete operators lead to different values of the Riccati coefficients \(a\), \(b\), and \(c\). These coefficients control the resulting a priori upper bound on the error and, through the associated Riccati ODE, the predicted growth of that bound in time. Consequently, studying how \(a\), \(b\), and \(c\) depend on the numerical fluxes, split forms, and SBP operators can help guide method design toward improved error-growth behavior. Furthermore, the coefficient \(b\) inherits an explicit dependence on the magnitude of discrete solution gradients (e.g., \cref{eq:systems_split_form_phi_term2}) and on changes in the flux Jacobian through \(c_F\), thereby linking problem-dependent nonlinear structure to the error-growth mechanism.
\end{remark}

\section{Conclusions}\label{sec:conclusions}
In this paper, we establish convergence of entropy-stable split-form C-SBP discretizations for scalar hyperbolic PDEs and symmetric hyperbolic systems with homogeneous fluxes on periodic domains, under smoothness assumptions on the exact solution and flux functions and global boundedness of the second derivatives of the fluxes. In contrast to analyses that rely on linearization, frozen-coefficient arguments, or projection-based decompositions, we work directly with the semi-discrete error equation. In particular, we show that the error norm satisfies a nonlinear differential inequality that is dominated by a constant-coefficient Riccati initial-value problem. For sufficiently small mesh spacing, and for degree-$p$ C-SBP discretizations in $d$ spatial dimensions with $p>1+d/2$, the Riccati envelope exists on any fixed finite time interval and tends to zero under mesh refinement, which implies convergence of the numerical solution. The resulting framework also clarifies how the coefficients in the error evolution inequality control nonlinear error growth.

Future work includes extending the analysis to discontinuous SBP discretizations with SATs (numerical fluxes), and to symmetrizable systems for which the relevant energy estimates involve state-dependent symmetrizers. It is also of interest to apply the present approach to entropy-stable discretizations that do not rely on homogeneity-based splittings, such as the Nordstr{\"o}m-split method \cite{nordstrom2022skew} and the residual-correction approach of Abgrall \cite{abgrall2018general}. Finally, it would be useful to characterize schemes through the coefficients appearing in the Riccati-type error evolution bound, both to understand long-time error behavior and to guide the design of more robust and efficient discretizations.

\begin{appendices}
\crefalias{section}{appendix}
\Crefname{appendix}{Appendix}{Appendices}
\crefname{appendix}{Appendix}{Appendices}

\renewcommand{\theHequation}{\thesection.\arabic{equation}}
\renewcommand{\theHfigure}{\thesection.\arabic{figure}}
\renewcommand{\theHtable}{\thesection.\arabic{table}}

\section{Flux expansion and remainder bound}\label{app:remainder_bound}
Let \(f:\mathbb{R}\to\mathbb{R}\) be \(C^2(\mathbb{R})\) with globally bounded second derivative, and assume that the numerical solution is bounded, \(|u_{h,i}|< \infty\). Denote the pointwise error by
\[
    e_i \coloneqq u_i - u_{h,i}.
\]
We expand \(f(u_{h,i})=f(u_i-e_i)\) about \(u_i\) and write the first-order Taylor polynomial plus a remainder
\[
    f(u_{h,i}) = f(u_{i} - e_{i}) = f(u_i) - f'(u_i)e_i + R_1(e_i),
\]
where \(f'(u_{i})=\frac{\partial{f}}{\partial{u}}(u_i)\). Hence, we have
\(
    f(u_i) - f(u_{h,i}) = f'(u_i)e_i - R_1(e_i),
\)
or in vector form,
\begin{equation}\label{eq:flux_taylor_expansion}
    f(\bm{u}) - f(\bm{u}_{h}) = \A(\bm{u})\bm{e} - R_1(\bm{e}),
\end{equation}
where \(\A(\bm{u}) \coloneqq \text{diag}(f'(\bm{u}))\).
The Taylor remainder satisfies the inequality 
\[
    |R_{1}(e_{i})| \le \frac{c_r}{2}|u_{i} - u_{h,i}|^2 = \frac{c_r}{2}|e_{i}|^2,
\]
where \(c_r \coloneqq \max_{i} |f''(u_{\xi,i})|\) and \(u_{\xi,i} \in [u_{i}, u_{h,i}]\) for each global node \(i\), i.e., \(c_r\) is greater than or equal to the magnitude of \(f''(u_{\xi,i})\) at any nodes in the domain. 

\paragraph{\texorpdfstring{Bounding $R_1(\bm{e})$ in $2$-norm.}{Bounding R1(e) in 2-norm.}}
Using the global error definition in \cref{eq:error_def_system}, we can write
\begin{align}\label{eq:remainder_2_norm_bound0}
    \|R_1(\bm{e})\|_{2}^2 &=\sum_{\Omega_k\in\fn{T}_h} \|R_1(\bm{e}_{k})\|_2^2
    = \sum_{\Omega_k\in\fn{T}_h}\sum_{i=1}^{n_{p}} |R_1(e_{k,i})|^2
    \nonumber
    \\
    &\le \frac{c_r^2}{4}\sum_{\Omega_k\in\fn{T}_h}\sum_{i=1}^{n_{p}} e_{k,i}^4
    \le \frac{c_r^2}{4}\sum_{\Omega_k\in\fn{T}_h}\Big(\sum_{i=1}^{n_{p}} e_{k,i}^2\Big)^2
    \nonumber
    \\
    & = \frac{c_r^2}{4}\sum_{\Omega_k\in\fn{T}_h}\|\bm{e}_k\|_{2}^4
    \le \frac{c_r^2}{4}\Big(\sum_{\Omega_k\in\fn{T}_h}\|\bm{e}_k\|_{2}^2\Big)^2 = \frac{c_r^2}{4}\|\bm{e}\|_2^4,
\end{align}
where we have used \cref{eq:sum_of_product} to arrive at the second and third inequalities. Taking the square root of both sides, we obtain
\begin{equation}\label{eq:remainder_2_norm_bound1}
    \|R_1(\bm{e})\|_{2} \le \frac{c_r}{2}\|\bm{e}\|_{2}^2.
\end{equation}

\section{Lipschitz continuity of the flux Jacobian}\label{app:lipschitz_flux}
\begin{lemma}\label{lem:flux_lipschitz}
    Let $\fn{U}(\cdot,t)\in C^1(\Omega)$ for each $t\in[0,T]$ with $\Omega\subset\mathbb{R}$ compact, and let $\fn{F}\in C^2(\mathbb{R})$.
    Assume the derivatives are uniformly bounded on $\Omega\times[0,T]$, in particular
    \[
        M_u := \sup_{x\in\Omega,\,t\in[0,T]} \left|\pder[\fn{U}]{x}(x,t)\right| <\infty,\qquad
        M_{f''} := \sup_{\fn{S}\in \fn{U}(\Omega\times[0,T])} \left|\pder[^{2}\fn{F}]{\fn{U}^{2}}(\fn{S})\right| <\infty.
    \]
    Then for any $x_i,x_j\in\Omega$ and any $t\in[0,T]$,
    \[
        \left|\pder[\fn{F}]{\fn{U}}(\fn{U}(x_i,t)) - \pder[\fn{F}]{\fn{U}}(\fn{U}(x_j,t))\right|
        \le C\,|x_i-x_j|,
    \]
    where one may take $C=M_{f''}\,M_u$. In particular, $C$ is independent of the mesh spacing $h=|x_i-x_j|$.
\end{lemma}

\begin{proof}
    By the mean-value theorem applied to the composition \(x\mapsto \pder[\fn{F}]{\fn{U}}(\fn{U}(x,t))\), there exists \(\xi\) on the segment between \(x_i\) and \(x_j\) such that
    \[
        \pder[\fn{F}]{\fn{U}}(\fn{U}(x_i,t)) - \pder[\fn{F}]{\fn{U}}(\fn{U}(x_j,t)) = \left[\pder[^{2}\fn{F}]{\fn{U}^{2}}\big(\fn{U}(\xi,t)\big)\,\pder[\fn{U}]{x}(\xi,t)\right]\,(x_i-x_j).
    \]
    Taking absolute values and using the uniform bounds gives
    \[
        \left|\pder[\fn{F}]{\fn{U}}(\fn{U}(x_i,t)) - \pder[\fn{F}]{\fn{U}}(\fn{U}(x_j,t))\right|
        \le M_{f''}\,M_u\,|x_i-x_j| = C\,|x_i-x_j|,
    \]
    which proves the result. 
\end{proof}
Lemma~\ref{lem:flux_lipschitz} is stated for $\Omega\subset\mathbb{R}$.
In $\Omega\subset\mathbb{R}^d$, the same estimate holds with
$\sup\|\nabla \fn{U}\|$ in place of $M_u$ and $\|\bm{x}_i-\bm{x}_j\|$ in place of $|x_i-x_j|$.

\section{Sum of products lemmas}
For convenience, we state without proof two standard lemmas used repeatedly in the main text.
\begin{lemma}\label{lem:sum_of_product}
    Let \(a_{i}\ge 0,b_{i}\ge 0\) for \(i \in \{1,\dots,n\}\). Then the sum of their products is bounded by the product of their sums, \ie,
    \begin{equation}\label{eq:sum_of_product}
        \sum_{i}(a_{i}b_{i}) \le \Big(\sum_{i} a_{i}\Big) \Big(\sum_{i} b_{i}\Big).
    \end{equation}
\end{lemma}

\begin{lemma}\label{lem:sum_of_H_norm_products}
    The sum of products of \(\H\)-norms of vectors is bounded by the products of the \(\H\)-norms of the individual vectors, \ie,
    \begin{equation}\label{eq:sum_of_H_norm_products}
        \sum_{\Omega_{k}\in \fn{T}_{h}} \|\bm{e}_k\|_{\H}\|\bm{\tau}_k\|_{\H} \le \Big(\sum_{\Omega_{k}\in \fn{T}_{h}} \|\bm{e}_k\|_{\H}^2\Big)^{1/2} \Big(\sum_{\Omega_{k}\in \fn{T}_{h}} \|\bm{\tau}_k\|_{\H}^2\Big)^{1/2} = \|\bm{e}\|_{\H}\|\bm{\tau}\|_{\H}.
    \end{equation}
\end{lemma}

\section{Solution of the constant-coefficient Riccati ODE}\label{app:riccati}
For completeness, we derive the explicit solution of the constant-coefficient Riccati ODE. Consider the scalar initial-value problem
\begin{equation}\label{eq:riccati}
    \pder[y(t)]{t} = y'(t) =  a\,y(t)^2 + b\,y(t) + c,\qquad y(0)=0, \qquad a \ge 0, \quad b\ge 0, \quad c > 0,
\end{equation}
and denote the discriminant
\(
    \Delta := b^2 - 4 a c.
\)
We do not consider the \(c=0\) case because, as discussed in the proof of \cref{lem:riccati_no_blowup_smallh}, the solution to the resulting ODE is the trivial solution, \(y=0\). In the following, we derive explicit formulas for \(y(t)\) by integrating the separable ODE on the finite interval \([0,t]\). We first treat the degenerate cases in which the quadratic term is absent, before turning to the general case \(a>0\).

\subsection{Degenerate cases}

\paragraph{Case 1. \(a=0\) and \(b=0\).}
The equation reduces to \(y'(t)=c\), \(y(0)=0\). Integrating directly gives
\begin{equation}\label{eq:case1_solution}
    y(t)=ct.
\end{equation}

\paragraph{Case 2. \(a=0\) and \(b>0\).}
The equation becomes \(y'(t)=by(t)+c\), \(y(0)=0\). Separating variables yields
\(
    \dd t={\dd y}/({by+c}).
\)
Integrating from \(0\) to \(t\) gives
\(
    t=\frac{1}{b}\ln \bigl(\frac{by(t)+c}{c}\bigr),
\)
which implies
\begin{equation}\label{eq:case2_solution}
    y(t)=\frac{c}{b}\bigl(e^{bt}-1\bigr).
\end{equation}

\paragraph{Case 3. \(a>0\), \(b=0\), and \(c>0\).}
The equation reduces to
\(
    y'(t)=ay(t)^2+c,
\)
\(y(0)=0\). Separating variables gives
\(
    \dd t={\dd y}/({ay^2+c}).
\)
Integrating from \(0\) to \(t\) yields
\(
    t=\frac{1}{\sqrt{ac}}\tan^{-1} \bigl(\sqrt{\frac{a}{c}}\,y(t)\bigr),
\)
and hence
\begin{equation}\label{eq:case3_solution}
    y(t)=\sqrt{\frac{c}{a}}\tan\!\left(\sqrt{ac}\,t\right).
\end{equation}

\subsection{General case, \(a>0\)}
We first note a few properties:
\begin{itemize}
    \item We assume \(y(t)\ge0\) since we are interested in the case where \(y(t)\ge\|\bm{e}(t)\|_{\H}\).
    \item Since \(a\neq 0\) the quadratic polynomial factorizes as
          \(
              a y^2 + b y + c = a (y-r_1)(y-r_2),
          \)
          and without loss of generality, we assume \(r_{1} = ({-b + \sqrt{\Delta}})/({2a})\), \(r_{2}= ({-b - \sqrt{\Delta}})/({2a})\), and thus \( r_1-r_2 = {\sqrt{\Delta}}/{a}\) which implies \(a(r_1-r_2)=\sqrt{\Delta}\).
    \item There are three cases: when \(\Delta>0\) the roots are real and distinct; when \(\Delta=0\) the root is double; when \(\Delta<0\) they are complex conjugates.
\end{itemize}

\paragraph{Case 1.} \(\Delta>0\) (two distinct real roots). Write the ODE as \(y'(t) = a(y-r_1)(y-r_2)\) and integrate from \(0\) to \(t\):
\begin{align*}
    t
     & = \int_{0}^{t} 1\,\dd s
    = \int_{0}^{t} \frac{y'(s)}{a(y(s)-r_1)(y(s)-r_2)}\,\dd s
    = \int_{y(0)}^{y(t)} \frac{\dd y}{a(y-r_1)(y-r_2)}.
\end{align*}
We apply a partial-fraction decomposition to write
\[
    \frac{1}{a(y-r_1)(y-r_2)}
    = \frac{1}{a(r_1-r_2)}\Big(\frac{1}{y-r_1}-\frac{1}{y-r_2}\Big).
\]
Integrating and using \(y(0)=0\), we obtain
\begin{align*}
    t
     & = \frac{1}{a(r_1-r_2)}\int_{0}^{y(t)}\Big(\frac{1}{y-r_1}-\frac{1}{y-r_2}\Big)\dd y 
     \\  
     & = \frac{1}{a(r_1-r_2)}\Big(\ln|y(t)-r_1| - \ln|y(t)-r_2|
    - \ln|-r_1| + \ln|-r_2|\Big)                                                              
    \\
     & = \frac{1}{a(r_1-r_2)}\ln\!\Bigg(\frac{y(t)-r_1}{y(t)-r_2}\;\frac{-r_2}{-r_1}\Bigg)
    = \frac{1}{\sqrt{\Delta}}\ln\!\Bigg(\frac{y(t)-r_1}{y(t)-r_2}\;\frac{r_2}{r_1}\Bigg),
\end{align*}
where we used \(a(r_1-r_2)=\sqrt{\Delta}\) and, since \(a>0\), \(b\ge 0\), and \(c>0\) imply \(r_1<0\) and \(r_2<0\) for \(\Delta>0\), all logarithm arguments are positive for \(y(t)\ge 0\) and the absolute values may be dropped. Exponentiating both sides,
multiplying through by \(r_1/r_2\), and rearranging, we find
\[
    \frac{y(t)-r_1}{y(t)-r_2} = \frac{r_1}{r_2}\,e^{\sqrt{\Delta}\,t}
    \quad\Longrightarrow\quad
    y(t)\Big(1-\frac{r_1}{r_2}e^{\sqrt{\Delta}\,t}\Big)
    = r_1\big(1-e^{\sqrt{\Delta}\,t}\big).
\]
Thus, provided the denominator is nonzero, we have
\begin{equation}\label{eq:solution_Dpos}
    y(t)=\dfrac{r_1\big(1-e^{\sqrt{\Delta}\,t}\big)}{\,1-(r_1/r_2)e^{\sqrt{\Delta}\,t}\,}.
\end{equation}

\paragraph{Case 2.} \(\Delta=0\) (double root). If \(\Delta=0\), then \(c=b^2/(4a)\) and \(r_1=r_2 = r=-b/(2a)\), thus the ODE becomes
\(
    y' = a(y-r)^2.
\)
Integrating from \(0\) to \(t\), we have
\[
    t = \int_{0}^{y(t)}\frac{\dd y}{a(y-r)^2}
    = \Big[-\frac{1}{a(y-r)}\Big]_{0}^{y(t)}
    = -\frac{1}{a(y(t)-r)} + \frac{1}{a(0-r)}
    = -\frac{1}{a(y(t)-r)} - \frac{1}{ar}.
\]
Rearranging, we find
\begin{equation}\label{eq:solution_Dzero}
    y(t)= r - \frac{1}{a t + 1/r}.
\end{equation}

\paragraph{Case 3.} \(\Delta<0\) (complex conjugate roots). Let \(\omega := \sqrt{4ac-b^2} > 0\) and define
\[
    \alpha := -\frac{b}{2a},\qquad \beta := \frac{\omega}{2a},
    \qquad a y^2 + b y + c = a\big((y-\alpha)^2 + \beta^2\big).
\]
Then the ODE can be written as \(y' = a\big((y-\alpha)^2 + \beta^2\big)\). Integrating from \(0\) to \(t\), we obtain
\begin{align*}
    t
     & = \int_{0}^{y(t)} \frac{\dd y}{a\big((y-\alpha)^2+\beta^2\big)}
    = \frac{1}{a\beta}\left[\tan^{-1}\left(\frac{y-\alpha}{\beta}\right)\right]_{0}^{y(t)}                                            
    \\
     & = \frac{1}{a\beta}\left(\tan^{-1}\left(\frac{y(t)-\alpha}{\beta}\right) - \tan^{-1}\left(\frac{-\alpha}{\beta}\right)\right).
\end{align*}
Since \(a\beta=\omega/2\), multiplying through and rearranging gives
\[
    \tan^{-1}\left(\frac{y(t)-\alpha}{\beta}\right) = \frac{\omega}{2}t + \tan^{-1}\left(\frac{b}{\omega}\right),
\]
where we have used the fact that \(\tan^{-1}(x)\) is an odd function. Applying the tangent function on both sides, we have
\[
    \frac{y(t)-\alpha}{\beta}=\tan\left(\frac{\omega}{2}t + \tan^{-1}\left(\frac{b}{\omega}\right)\right),
\]
which gives the real-valued explicit solution as
\begin{align}
    y(t) & =\alpha + \beta \tan\left(\frac{\omega}{2}t + \tan^{-1}\left(\frac{b}{\omega}\right)\right) =\frac{\omega \tan\big(\frac{\omega}{2}t+\tan^{-1}(b/\omega)\big)-b}{2a},
    \nonumber
\end{align}
which, in terms of \(\Delta\) is
\begin{align}
    y(t) & =\frac{\sqrt{-\Delta} \tan\big(\frac{\sqrt{-\Delta}}{2}t+\tan^{-1}(b/\sqrt{-\Delta})\big)-b}{2a}.\label{eq:solution_Dneg}
\end{align}
\section{Blow-up time for the Riccati ODE: Proof of \texorpdfstring{\cref{prop:csbp_error_estimate_divergence_form}}{Proposition~\ref{prop:csbp_error_estimate_divergence_form}}}\label{sec:riccati_blow_up_time}
\begin{proof}
    The proof uses the error equation, \cref{eq:riccati_energy_error}, for the entropy-stable C-SBP discretization, which is reproduced here as
    \[
        \der[y]{t} = y' = ay^{2} + by + c, \qquad y(0) = 0.
    \]
    Recall that \(a\ge0\), \(b\ge0\), and \(c>0\) are constant coefficients. Furthermore, if \(c=0\), then the solution to the ODE is \(y=0\), and thus the blow-up time is \(t_{*}=\infty\). We define \(\Delta \coloneqq b^2 - 4ac\), and consider six cases.
    
    \paragraph{Case 1.} \(a=0\) and \(b=0\). The solution for this case is given by \cref{eq:case1_solution}, 
    \begin{equation}
        y(T)=c T,
    \end{equation}
    \ie, at most, the error grows linearly in time.
    
    \paragraph{Case 2.} \(a=0\) and \(b > 0\). In this case the solution is given by \cref{eq:case2_solution},
    \begin{equation}
        y(T) = \dfrac{c}{b}\big(e^{b T}-1\big).
    \end{equation}
    
    \paragraph{Case 3.} \(a>0\), \(b=0\), and \(c>0\).
    In this case, we have the solution given by \cref{eq:case3_solution},
    \(
        y(T) = \sqrt{\frac{c}{a}} \, \tan\left(\sqrt{ac}\,T\right).
    \)
    The tangent becomes unbounded when its argument is \(\frac{\pi}{2} + k\pi\); thus, the earliest blow-up occurs when
    \(
        \sqrt{ac}\,t_* = {\pi}/{2},
    \)
    that is,
    \begin{equation}\label{eq:case3_blowup}
        t_* = \frac{\pi}{2\sqrt{ac}}.
    \end{equation}
    
    \paragraph{Case 4.} \(a > 0\) and \(b > 0\), and \(\Delta=b^2 - 4ac > 0\). In this case, the error equation, \cref{eq:riccati_energy_error}, coincides with the constant-coefficient Riccati equation, the solutions of which are given in \cref{app:riccati}. The solution at \(t=T\) is given by \cref{eq:solution_Dpos},
    \begin{equation}\label{eq:case4_solution_Dpos1}
        y(T)=\dfrac{r_1\big(1-e^{\sqrt{\Delta}\,T}\big)}{\,1-(r_1/r_2)e^{\sqrt{\Delta}T}},
    \end{equation}
    where \(r_{1} = ({-b + \sqrt{\Delta}})/({2a})\) and \(r_{2}=({-b - \sqrt{\Delta}})/({2a})\).
    Now, the explicit solution \eqref{eq:case4_solution_Dpos1} is well-defined on \([0,T]\) if and only if
    \(
        1-({r_1}/{r_2})e^{\sqrt{\Delta}\,t}\neq 0
    \)
    for all \(t\in[0,T]\).
    If the denominator vanishes at some \(t_*>0\) then the solution blows up at \(t_*\).  The earliest blow-up time is, therefore,
    \begin{equation}\label{eq:case4_blow_up}
        t_*=\frac{1}{\sqrt{\Delta}}\ln \left(\frac{r_2}{r_1}\right) = \frac{1}{\sqrt{\Delta}}\ln\left(\frac{-b - \sqrt{\Delta}}{-b + \sqrt{\Delta}}\right).
    \end{equation}
    Note that \(\sqrt{\Delta} = \sqrt{b^2-4ac} < \sqrt{b^2} = b\); hence, we always have \(r_{2}/r_{1} > 1\) and thus \(t_*>0\).
    
    \paragraph{Case 5.} \(a > 0\), \(b > 0\),  and \(\Delta=b^2 - 4ac = 0\). In this case, the solution to the Riccati equation is given by \cref{eq:solution_Dzero},
    \begin{equation}\label{eq:case5_solution_Dzero1}
        y(T)= r - \frac{1}{a T + 1/r} = \frac{b}{2a} \left(\frac{2}{2-bT} - 1\right),
    \end{equation}
    where \(r=-\frac{b}{2a}\).
    Existence of the Riccati solution on \([0,T]\) requires \(a t + 1/r\neq 0\) for all \(t\in[0,T]\). This condition gives the solution blow-up time as
    \begin{equation}\label{eq:case5_blow_up}
        t_{*} = -\frac{1}{ar} = \frac{2}{b}.
    \end{equation}
    
    \paragraph{Case 6.} \(a > 0\), \(b > 0\),  and \(\Delta=b^2 - 4ac < 0\). In this case, the solution to the Riccati equation is given by \cref{eq:solution_Dneg},
    \begin{equation}\label{eq:case6_solution_Dneg1}
        y(T) =\frac{\sqrt{-\Delta} \tan\big(\frac{\sqrt{-\Delta}}{2}T+\tan^{-1}(b/\sqrt{-\Delta})\big)-b}{2a}.
    \end{equation}
    Existence of the Riccati solution on \([0,T]\) requires that the argument of the tangent function does not hit a pole \(\frac{\pi}{2} + k\pi\) for any \(t\in[0,T]\). Hence, the smallest blow-up time in this case satisfies \(({\sqrt{-\Delta}}/{2}) t_{*}+\tan^{-1}(b/\sqrt{-\Delta}) = {\pi}/{2}\), which gives the blow-up time as
    \begin{equation}\label{eq:case6_blow_up}
        t_{*} = \frac{\pi - 2\tan^{-1}(b/\sqrt{-\Delta})}{\sqrt{-\Delta}}.
    \end{equation}
    Note that \(\tan^{-1}(b/\sqrt{-\Delta})\le \frac{\pi}{2}\); hence, \(t_{*}\) is positive.
\end{proof}

\section{SBP operator mesh scaling}\label{app:sbp_operators}
On the reference element \(\hat\Omega\), an SBP operator satisfies, for each coordinate \(\xi_i\),
\[
    \D_{\xi_i}=\H^{-1}\Q_{\xi_i},\qquad
    \Q_{\xi_i}=\mathsf{S}_{\xi_i}+\frac{1}{2}\E_{\xi_i},
\]
with a symmetric positive definite \(\H\) matrix, \(\mathsf{S}_{\xi_i}=-\mathsf{S}_{\xi_i}^T\), and \(\Q_{\xi_i}+\Q_{\xi_i}^T=\E_{\xi_i}\). We assume that \(\H\) is diagonal. To study the scaling of SBP matrices on a given physical element, \(\Omega_{k}\), we make the following shape-regularity assumption. 
\begin{assumption}\label{assump:shape_regular_mapping}
For each element \(\Omega_k\), let \(h_k := \operatorname{diam}(\Omega_k)\), and let
\(\bm{x}_k : \hat\Omega \to \Omega_k\) be a smooth invertible mapping with Jacobian
\[
    \J_k(\bm{\xi}) := \frac{\partial \bm{x}_k}{\partial \bm{\xi}}(\bm{\xi}).
\]
Assume the family of mappings is shape regular in the sense that there exist constants
\(c_J, C_J > 0\), independent of \(k\) and \(h\), such that for all
\(\bm{\xi}\in\hat\Omega\),
\[
    c_J h_k \le \sigma_{\min}\!\bigl(\J_k(\bm{\xi})\bigr)
    \le \sigma_{\max}\!\bigl(\J_k(\bm{\xi})\bigr)
    \le C_J h_k.
\]
Equivalently,
\(
    \|\J_k(\bm{\xi})\|_2 \le C_J h_k
\)
and
\(
    \|\J_k(\bm{\xi})^{-1}\|_2 \le c_J^{-1} h_k^{-1}
\)
for all \(\bm{\xi}\in\hat\Omega\).
\end{assumption}

Using the mapping in \cref{assump:shape_regular_mapping}, and denoting the outward pointing normal vector on facet \(\gamma\in\hat{\Gamma}\) of \(\hat{\Omega}\) by \(n_{\xi_j\gamma}\) and the determinant of the metric Jacobian at node \(m\) by \(|\J_k^{(m)}|\), the SBP matrices on the physical elements are constructed as \cite{crean2018entropy}
\begin{align*}
    [\mathsf{S}_{x_i k}]_{mn}
                & = \frac{1}{2}\sum_{j=1}^d \Bigl(
    |\J_k^{(m)}|\frac{\partial \xi_j}{\partial x_i}(\bm{\xi}^{(m)})[\Q_{\xi_j}]_{mn}
    - |\J_k^{(n)}|\frac{\partial \xi_j}{\partial x_i}(\bm{\xi}^{(n)})[\Q_{\xi_j}]_{nm}\Bigr),
   \\
    [\H_k]_{mm} & = [\H]_{mm}\,|\J_k^{(m)}|, \qquad 
    [\N_{x_i\gamma k}]_{mm}
        = |\J_k^{(m)}|\sum_{j=1}^d n_{\xi_j\gamma}(\bm{\xi}^{(m)})
        \frac{\partial \xi_j}{\partial x_i}(\bm{\xi}^{(m)}),
    \\
    \E_{x_i k}  &= \sum_{\gamma\in\Gamma_k} \R_{\gamma k}^T\,\B_\gamma\,\N_{x_i\gamma k}\,\R_{\gamma k},
    \qquad
    \Q_{x_i k}  = \mathsf{S}_{x_i k}+\frac{1}{2}\E_{x_i k},
    \qquad
    \D_{x_i k}  = \H_k^{-1}\Q_{x_i k}.
\end{align*}
\begin{lemma}\label{lem:sbp_scaling}
Let \cref{assump:shape_regular_mapping,assump:mesh} hold. Then, for each element \(\Omega_k\), each facet \(\gamma\in\Gamma_k\), and each node \(m\) for nodewise quantities, we have
\begin{align*}
    |\J_k^{(m)}|           & = \fn{O}(h^d),
                           &                & \|\H_k\| = \fn{O}(h^d),
    \\
    \|\N_{x_i\gamma k}\|   & = \fn{O}(h^{d-1}),
                           &                & \|\mathsf{S}_{x_i k}\| = \fn{O}(h^{d-1}),
                           &                & \|\E_{x_i k}\| = \fn{O}(h^{d-1}),
                           &                & \|\Q_{x_i k}\| = \fn{O}(h^{d-1}),
    \\
    \|\D_{x_i k}\|         & = \fn{O}(h^{-1}),
\end{align*}
where \(\|\cdot\|\) denotes any induced norm.
\end{lemma}
\begin{proof}
    Since the spatial dimension, \(d\), the number of volume nodes, \(n_p\), and the number of facet nodes, \(n_f\) are fixed, any induced matrix norm is equivalent to \(\|\cdot\|_2\); therefore, the scaling estimates below are equivalent, up to a constant, for any induced norm. Furthermore, if each entry of a local element or facet matrix is \(\fn{O}(h^\alpha)\), then its matrix norm is also \(\fn{O}(h^\alpha)\). 
    \textbf{Scaling of \(|\J_k^{(m)}|\).}
    By \cref{assump:shape_regular_mapping}, the singular values of \(\J_k(\bm{\xi})\) satisfy
    \[
        c_J h_k \le \sigma_{\min}\!\bigl(\J_k(\bm{\xi})\bigr)
        \le \sigma_{\max}\!\bigl(\J_k(\bm{\xi})\bigr) \le C_J h_k,
        \qquad \bm{\xi}\in\hat\Omega.
    \]
    Hence, at each element node \(\bm{\xi}^{(m)}\),
    \[
        |\J_k^{(m)}|
        =
        \bigl|\det(\J_k(\bm{\xi}^{(m)}))\bigr|
        =
        \prod_{r=1}^d \sigma_r\!\bigl(\J_k(\bm{\xi}^{(m)})\bigr)
        =
        \fn{O}(h_k^d).
    \]

    \medskip\noindent
    \textbf{Scaling of \(\H_k\).}
    Since \([\H_k]_{mm} = [\H]_{mm}\,|\J_k^{(m)}|\) with \([\H]_{mm}=\fn{O}(1)\), we obtain
    \(
        \|\H_k\|_2 = \fn{O}(h_k^d).
    \)
    Moreover, the lower bound on the singular values of \(\J_k(\bm{\xi})\) implies
    \(
        |\J_k^{(m)}|^{-1} = \fn{O}(h_k^{-d}),
    \)
    and therefore
    \(
        \|\H_k^{-1}\|_2 = \fn{O}(h_k^{-d}).
    \)

    \medskip\noindent
    \textbf{Scaling of \(\N_{x_i\gamma k}\).}
    Let \(\bm{\nu}_i\in\mathbb{R}^d\) denote the \(i\)-th column of the identity matrix, \(\I_d\). Then, for \(i,j\in\{1,\dots,d\}\), the Cauchy--Schwarz inequality gives
    \begin{equation*}
        \left|\frac{\partial \xi_j}{\partial x_i}(\bm{\xi}^{(m)})\right|
        =
        \left|\bm{\nu}_{j}^T\J_k\left(\bm{\xi}^{(m)}\right)^{-1}\bm{\nu}_{i}\right|
        \le
        \|\bm{\nu}_{j}\|\,\left\|\J_k(\bm{\xi}^{(m)})^{-1}\right\|\,\|\bm{\nu}_{i}\|
        =
        \fn{O}(h_k^{-1}).
    \end{equation*}
    Noting the terms in \(\N_{x_i\gamma k}\) and combining \(|\J_k^{(m)}|=\fn{O}(h_k^d)\) and \(|\partial\xi/\partial x|=\fn{O}(h_k^{-1})\) gives
    \(
        [\N_{x_i\gamma k}]_{mm}=\fn{O}(h_k^{d-1}),
    \)
    and hence
    \(
        \|\N_{x_i\gamma k}\|_2=\fn{O}(h_k^{d-1}).
    \)

    \medskip\noindent
    \textbf{Scaling of \(\E_{x_i k}\).}
    By construction
    \(
        \E_{x_i k}=\sum_{\gamma\in\Gamma_{k}} \R_{\gamma k}^T\B_\gamma \N_{x_i\gamma k}\R_{\gamma k},
    \)
    where \(\R_{\gamma k}\) and \(\B_\gamma\) are independent of \(h\), thus it follows that
    \(
        \|\E_{x_i k}\|_2=\fn{O}(h_k^{d-1}).
    \)

    \medskip\noindent
    \textbf{Scaling of \(\mathsf{S}_{x_i k}\) and \(\Q_{x_i k}\).}
    Each term in \(\mathsf{S}_{x_i k}\) is of the form
    \(
        |\J_k^{(m)}|
        \frac{\partial \xi_j}{\partial x_i}(\bm{\xi}^{(m)})
        [\Q_{\xi_j}]_{mn}.
    \)
    Since \(|\J_k^{(m)}|=\fn{O}(h_k^d)\), \(|\partial\xi/\partial x|=\fn{O}(h_k^{-1})\), and \([\Q_{\xi_j}]_{mn}=\fn{O}(1)\), each entry of \(\mathsf{S}_{x_i k}\) is \(\fn{O}(h_k^{d-1})\). The number of local nodes, \(n_{p}\), is fixed, thus we have
    \(
        \|\mathsf{S}_{x_i k}\|_2=\fn{O}(h_k^{d-1}).
    \)
    Together with \(\|\E_{x_i k}\|_2=\fn{O}(h_k^{d-1})\), we conclude that
    \(
        \|\Q_{x_i k}\|_2=\fn{O}(h_k^{d-1}).
    \)

    \medskip\noindent
    \textbf{Scaling of \(\D_{x_i k}\).}
    Since \(\D_{x_i k}=\H_k^{-1}\Q_{x_i k}\), with
    \(
        \|\H_k^{-1}\|_2=\fn{O}(h_k^{-d})
    \)
    and
    \(
        \|\Q_{x_i k}\|_2=\fn{O}(h_k^{d-1}),
    \)
    we have
    \(
        \|\D_{x_i k}\|_2=\fn{O}(h_k^{-1}).
    \)
    
    Finally, since there exist \(c_1,C_1,c_2,C_2>0\), independent of \(k\), \(h\), and \(m\), such that \(c_1 h_k^d \le [\H_k]_{mm} \le C_1 h_k^d\) and, by \cref{assump:mesh}, \(c_2 h^d \le [\H_k]_{mm}=w_{k,m} \le C_2 h^d\), it follows that \(h_k\) and \(h\) are equivalent, and thus every \(\fn{O}(h_k^\alpha)\) bound above may be written as \(\fn{O}(h^\alpha)\).
\end{proof}

\end{appendices}


\bibliography{references}


\begin{thebibliography}{27}
\ifx \bisbn   \undefined \def \bisbn  #1{ISBN #1}\fi
\ifx \binits  \undefined \def \binits#1{#1}\fi
\ifx \bauthor  \undefined \def \bauthor#1{#1}\fi
\ifx \batitle  \undefined \def \batitle#1{#1}\fi
\ifx \bjtitle  \undefined \def \bjtitle#1{#1}\fi
\ifx \bvolume  \undefined \def \bvolume#1{\textbf{#1}}\fi
\ifx \byear  \undefined \def \byear#1{#1}\fi
\ifx \bissue  \undefined \def \bissue#1{#1}\fi
\ifx \bfpage  \undefined \def \bfpage#1{#1}\fi
\ifx \blpage  \undefined \def \blpage #1{#1}\fi
\ifx \burl  \undefined \def \burl#1{\textsf{#1}}\fi
\ifx \doiurl  \undefined \def \doiurl#1{\url{https://doi.org/#1}}\fi
\ifx \betal  \undefined \def \betal{\textit{et al.}}\fi
\ifx \binstitute  \undefined \def \binstitute#1{#1}\fi
\ifx \binstitutionaled  \undefined \def \binstitutionaled#1{#1}\fi
\ifx \bctitle  \undefined \def \bctitle#1{#1}\fi
\ifx \beditor  \undefined \def \beditor#1{#1}\fi
\ifx \bpublisher  \undefined \def \bpublisher#1{#1}\fi
\ifx \bbtitle  \undefined \def \bbtitle#1{#1}\fi
\ifx \bedition  \undefined \def \bedition#1{#1}\fi
\ifx \bseriesno  \undefined \def \bseriesno#1{#1}\fi
\ifx \blocation  \undefined \def \blocation#1{#1}\fi
\ifx \bsertitle  \undefined \def \bsertitle#1{#1}\fi
\ifx \bsnm \undefined \def \bsnm#1{#1}\fi
\ifx \bsuffix \undefined \def \bsuffix#1{#1}\fi
\ifx \bparticle \undefined \def \bparticle#1{#1}\fi
\ifx \barticle \undefined \def \barticle#1{#1}\fi
\bibcommenthead
\ifx \bconfdate \undefined \def \bconfdate #1{#1}\fi
\ifx \botherref \undefined \def \botherref #1{#1}\fi
\ifx \url \undefined \def \url#1{\textsf{#1}}\fi
\ifx \bchapter \undefined \def \bchapter#1{#1}\fi
\ifx \bbook \undefined \def \bbook#1{#1}\fi
\ifx \bcomment \undefined \def \bcomment#1{#1}\fi
\ifx \oauthor \undefined \def \oauthor#1{#1}\fi
\ifx \citeauthoryear \undefined \def \citeauthoryear#1{#1}\fi
\ifx \endbibitem  \undefined \def \endbibitem {}\fi
\ifx \bconflocation  \undefined \def \bconflocation#1{#1}\fi
\ifx \arxivurl  \undefined \def \arxivurl#1{\textsf{#1}}\fi
\csname PreBibitemsHook\endcsname

\bibitem[\protect\citeauthoryear{Hicken et~al.}{2016}]{hicken2016multidimensional}
\begin{barticle}
\bauthor{\bsnm{Hicken}, \binits{J.E.}},
\bauthor{\bsnm{Del Rey~Fern{\'a}ndez}, \binits{D.C.}},
\bauthor{\bsnm{Zingg}, \binits{D.W.}}:
\batitle{Multidimensional summation-by-parts operators: General theory and application to simplex elements}.
\bjtitle{SIAM Journal on Scientific Computing}
\bvolume{38}(\bissue{4}),
\bfpage{1935}--\blpage{1958}
(\byear{2016})
\end{barticle}
\endbibitem

\bibitem[\protect\citeauthoryear{Hicken}{2020}]{hicken2020entropy}
\begin{barticle}
\bauthor{\bsnm{Hicken}, \binits{J.E.}}:
\batitle{Entropy-stable, high-order summation-by-parts discretizations without interface penalties}.
\bjtitle{Journal of Scientific Computing}
\bvolume{82}(\bissue{2}),
\bfpage{50}
(\byear{2020})
\end{barticle}
\endbibitem

\bibitem[\protect\citeauthoryear{Strang}{1964}]{strang1964accurate}
\begin{barticle}
\bauthor{\bsnm{Strang}, \binits{G.}}:
\batitle{Accurate partial difference methods {II: Non-linear} problems}.
\bjtitle{Numerische Mathematik}
\bvolume{6}(\bissue{1}),
\bfpage{37}--\blpage{46}
(\byear{1964})
\end{barticle}
\endbibitem

\bibitem[\protect\citeauthoryear{Gassner et~al.}{2022}]{gassner2022stability}
\begin{barticle}
\bauthor{\bsnm{Gassner}, \binits{G.J.}},
\bauthor{\bsnm{Sv{\"a}rd}, \binits{M.}},
\bauthor{\bsnm{Hindenlang}, \binits{F.J.}}:
\batitle{Stability issues of entropy-stable and/or split-form high-order schemes}.
\bjtitle{Journal of Scientific Computing}
\bvolume{90}(\bissue{3}),
\bfpage{1}--\blpage{36}
(\byear{2022})
\end{barticle}
\endbibitem

\bibitem[\protect\citeauthoryear{Kreiss and Lorenz}{1989}]{kreiss1989initial}
\begin{bbook}
\bauthor{\bsnm{Kreiss}, \binits{H.-O.}},
\bauthor{\bsnm{Lorenz}, \binits{J.}}:
\bbtitle{Initial-Boundary Value Problems and the {Navier-Stokes} Equations}
vol. \bseriesno{136}.
\bpublisher{Academic Press},
\blocation{San Diego, CA}
(\byear{1989})
\end{bbook}
\endbibitem

\bibitem[\protect\citeauthoryear{Mishra and Sv{\"a}rd}{2010}]{mishra2010stability}
\begin{barticle}
\bauthor{\bsnm{Mishra}, \binits{S.}},
\bauthor{\bsnm{Sv{\"a}rd}, \binits{M.}}:
\batitle{On stability of numerical schemes via frozen coefficients and the magnetic induction equations}.
\bjtitle{BIT Numerical Mathematics}
\bvolume{50}(\bissue{1}),
\bfpage{85}--\blpage{108}
(\byear{2010})
\end{barticle}
\endbibitem

\bibitem[\protect\citeauthoryear{Zhang and Shu}{2004}]{zhang2004error}
\begin{barticle}
\bauthor{\bsnm{Zhang}, \binits{Q.}},
\bauthor{\bsnm{Shu}, \binits{C.-W.}}:
\batitle{Error estimates to smooth solutions of {Runge--Kutta} discontinuous {Galerkin} methods for scalar conservation laws}.
\bjtitle{SIAM Journal on Numerical Analysis}
\bvolume{42}(\bissue{2}),
\bfpage{641}--\blpage{666}
(\byear{2004})
\end{barticle}
\endbibitem

\bibitem[\protect\citeauthoryear{Zhang and Shu}{2006}]{zhang2006error}
\begin{barticle}
\bauthor{\bsnm{Zhang}, \binits{Q.}},
\bauthor{\bsnm{Shu}, \binits{C.-W.}}:
\batitle{Error estimates to smooth solutions of {Runge--Kutta} discontinuous {Galerkin} method for symmetrizable systems of conservation laws}.
\bjtitle{SIAM Journal on Numerical Analysis}
\bvolume{44}(\bissue{4}),
\bfpage{1703}--\blpage{1720}
(\byear{2006})
\end{barticle}
\endbibitem

\bibitem[\protect\citeauthoryear{Huang and Shu}{2017}]{huang2017error}
\begin{barticle}
\bauthor{\bsnm{Huang}, \binits{J.}},
\bauthor{\bsnm{Shu}, \binits{C.-W.}}:
\batitle{Error estimates to smooth solutions of semi-discrete discontinuous {Galerkin} methods with quadrature rules for scalar conservation laws}.
\bjtitle{Numerical Methods for Partial Differential Equations}
\bvolume{33}(\bissue{2}),
\bfpage{467}--\blpage{488}
(\byear{2017})
\end{barticle}
\endbibitem

\bibitem[\protect\citeauthoryear{Nordstr{\"o}m}{2008}]{nordstrom2008error}
\begin{barticle}
\bauthor{\bsnm{Nordstr{\"o}m}, \binits{J.}}:
\batitle{Error bounded schemes for time-dependent hyperbolic problems}.
\bjtitle{SIAM Journal on Scientific Computing}
\bvolume{30}(\bissue{1}),
\bfpage{46}--\blpage{59}
(\byear{2008})
\end{barticle}
\endbibitem

\bibitem[\protect\citeauthoryear{Kopriva et~al.}{2017}]{kopriva2017error}
\begin{barticle}
\bauthor{\bsnm{Kopriva}, \binits{D.A.}},
\bauthor{\bsnm{Nordstr{\"o}m}, \binits{J.}},
\bauthor{\bsnm{Gassner}, \binits{G.J.}}:
\batitle{Error boundedness of discontinuous {Galerkin} spectral element approximations of hyperbolic problems}.
\bjtitle{Journal of Scientific Computing}
\bvolume{72}(\bissue{1}),
\bfpage{314}--\blpage{330}
(\byear{2017})
\end{barticle}
\endbibitem

\bibitem[\protect\citeauthoryear{{\"O}ffner and Ranocha}{2019}]{offner2019error}
\begin{barticle}
\bauthor{\bsnm{{\"O}ffner}, \binits{P.}},
\bauthor{\bsnm{Ranocha}, \binits{H.}}:
\batitle{Error boundedness of discontinuous {Galerkin} methods with variable coefficients}.
\bjtitle{Journal of Scientific Computing}
\bvolume{79}(\bissue{3}),
\bfpage{1572}--\blpage{1607}
(\byear{2019})
\end{barticle}
\endbibitem

\bibitem[\protect\citeauthoryear{Chen and Shu}{2017}]{chen2017entropy}
\begin{barticle}
\bauthor{\bsnm{Chen}, \binits{T.}},
\bauthor{\bsnm{Shu}, \binits{C.-W.}}:
\batitle{Entropy stable high order discontinuous {Galerkin} methods with suitable quadrature rules for hyperbolic conservation laws}.
\bjtitle{Journal of Computational Physics}
\bvolume{345},
\bfpage{427}--\blpage{461}
(\byear{2017})
\end{barticle}
\endbibitem

\bibitem[\protect\citeauthoryear{Worku et~al.}{2024}]{worku2024quadrature}
\begin{barticle}
\bauthor{\bsnm{Worku}, \binits{Z.A.}},
\bauthor{\bsnm{Hicken}, \binits{J.E.}},
\bauthor{\bsnm{Zingg}, \binits{D.W.}}:
\batitle{Quadrature rules on triangles and tetrahedra for multidimensional summation-by-parts operators}.
\bjtitle{Journal of Scientific Computing}
\bvolume{101}(\bissue{1}),
\bfpage{24}
(\byear{2024})
\end{barticle}
\endbibitem

\bibitem[\protect\citeauthoryear{Del Rey~Fern{\'a}ndez et~al.}{2014}]{fernandez2014review}
\begin{barticle}
\bauthor{\bsnm{Del Rey~Fern{\'a}ndez}, \binits{D.C.}},
\bauthor{\bsnm{Hicken}, \binits{J.E.}},
\bauthor{\bsnm{Zingg}, \binits{D.W.}}:
\batitle{Review of summation-by-parts operators with simultaneous approximation terms for the numerical solution of partial differential equations}.
\bjtitle{Computers \& Fluids}
\bvolume{95},
\bfpage{171}--\blpage{196}
(\byear{2014})
\end{barticle}
\endbibitem

\bibitem[\protect\citeauthoryear{Sv{\"a}rd and Nordstr{\"o}m}{2014}]{svard2014review}
\begin{barticle}
\bauthor{\bsnm{Sv{\"a}rd}, \binits{M.}},
\bauthor{\bsnm{Nordstr{\"o}m}, \binits{J.}}:
\batitle{Review of summation-by-parts schemes for initial--boundary-value problems}.
\bjtitle{Journal of Computational Physics}
\bvolume{268},
\bfpage{17}--\blpage{38}
(\byear{2014})
\end{barticle}
\endbibitem

\bibitem[\protect\citeauthoryear{Olsson and Oliger}{1994}]{olsson1994energy}
\begin{botherref}
\oauthor{\bsnm{Olsson}, \binits{P.}},
\oauthor{\bsnm{Oliger}, \binits{J.}}:
Energy and maximum norm estimates for nonlinear conservation laws.
Technical Report 94-01,
The Research Institute of Advanced Computer Science
(1994)
\end{botherref}
\endbibitem

\bibitem[\protect\citeauthoryear{Gerritsen and Olsson}{1996}]{gerritsen1996designing}
\begin{barticle}
\bauthor{\bsnm{Gerritsen}, \binits{M.}},
\bauthor{\bsnm{Olsson}, \binits{P.}}:
\batitle{Designing an efficient solution strategy for fluid flows: 1. {A} stable high order finite difference scheme and sharp shock resolution for the {Euler} equations}.
\bjtitle{Journal of Computational Physics}
\bvolume{129}(\bissue{2}),
\bfpage{245}--\blpage{262}
(\byear{1996})
\end{barticle}
\endbibitem

\bibitem[\protect\citeauthoryear{Yee et~al.}{2000}]{yee2000entropy}
\begin{barticle}
\bauthor{\bsnm{Yee}, \binits{H.C.}},
\bauthor{\bsnm{Vinokur}, \binits{M.}},
\bauthor{\bsnm{Djomehri}, \binits{M.J.}}:
\batitle{Entropy splitting and numerical dissipation}.
\bjtitle{Journal of Computational Physics}
\bvolume{162}(\bissue{1}),
\bfpage{33}--\blpage{81}
(\byear{2000})
\end{barticle}
\endbibitem

\bibitem[\protect\citeauthoryear{Sj{\"o}green and Yee}{2019}]{sjogreen2019entropy}
\begin{barticle}
\bauthor{\bsnm{Sj{\"o}green}, \binits{B.}},
\bauthor{\bsnm{Yee}, \binits{H.C.}}:
\batitle{Entropy stable method for the {Euler} equations revisited: central differencing via entropy splitting and {SBP}}.
\bjtitle{Journal of Scientific Computing}
\bvolume{81}(\bissue{3}),
\bfpage{1359}--\blpage{1385}
(\byear{2019})
\end{barticle}
\endbibitem

\bibitem[\protect\citeauthoryear{Worku and Zingg}{2024}]{worku2023entropy}
\begin{barticle}
\bauthor{\bsnm{Worku}, \binits{Z.A.}},
\bauthor{\bsnm{Zingg}, \binits{D.W.}}:
\batitle{Entropy-split multidimensional summation-by-parts discretization of the {Euler and compressible Navier-Stokes} equations}.
\bjtitle{Journal of Computational Physics}
\bvolume{502},
\bfpage{112821}
(\byear{2024})
\end{barticle}
\endbibitem

\bibitem[\protect\citeauthoryear{Johnson}{1990}]{johnson1990matrix}
\begin{bbook}
\bauthor{\bsnm{Johnson}, \binits{C.R.}}:
\bbtitle{Matrix Theory and Applications}
vol. \bseriesno{40}.
\bpublisher{American Mathematical Society},
\blocation{Providence, RI}
(\byear{1990})
\end{bbook}
\endbibitem

\bibitem[\protect\citeauthoryear{Bhatia}{1997}]{bhatia1996matrix}
\begin{bbook}
\bauthor{\bsnm{Bhatia}, \binits{R.}}:
\bbtitle{Matrix Analysis}.
\bsertitle{Graduate Texts in Mathematics},
vol. \bseriesno{169}.
\bpublisher{Springer},
\blocation{New York, NY}
(\byear{1997})
\end{bbook}
\endbibitem

\bibitem[\protect\citeauthoryear{Khalil}{2002}]{khalil2002nonlinear}
\begin{bbook}
\bauthor{\bsnm{Khalil}, \binits{H.K.}}:
\bbtitle{Nonlinear Systems},
\bedition{3}rd edn.
\bpublisher{Prentice Hall},
\blocation{Upper Saddle River, NJ}
(\byear{2002})
\end{bbook}
\endbibitem

\bibitem[\protect\citeauthoryear{Nordstr{\"o}m}{2022}]{nordstrom2022skew}
\begin{barticle}
\bauthor{\bsnm{Nordstr{\"o}m}, \binits{J.}}:
\batitle{A skew-symmetric energy and entropy stable formulation of the compressible {Euler} equations}.
\bjtitle{Journal of Computational Physics}
\bvolume{470},
\bfpage{111573}
(\byear{2022})
\end{barticle}
\endbibitem

\bibitem[\protect\citeauthoryear{Abgrall}{2018}]{abgrall2018general}
\begin{barticle}
\bauthor{\bsnm{Abgrall}, \binits{R.}}:
\batitle{A general framework to construct schemes satisfying additional conservation relations. {Application} to entropy conservative and entropy dissipative schemes}.
\bjtitle{Journal of Computational Physics}
\bvolume{372},
\bfpage{640}--\blpage{666}
(\byear{2018})
\end{barticle}
\endbibitem

\bibitem[\protect\citeauthoryear{Crean et~al.}{2018}]{crean2018entropy}
\begin{barticle}
\bauthor{\bsnm{Crean}, \binits{J.}},
\bauthor{\bsnm{Hicken}, \binits{J.E.}},
\bauthor{\bsnm{Del Rey~Fern{\'a}ndez}, \binits{D.C.}},
\bauthor{\bsnm{Zingg}, \binits{D.W.}},
\bauthor{\bsnm{Carpenter}, \binits{M.H.}}:
\batitle{Entropy-stable summation-by-parts discretization of the {E}uler equations on general curved elements}.
\bjtitle{Journal of Computational Physics}
\bvolume{356},
\bfpage{410}--\blpage{438}
(\byear{2018})
\end{barticle}
\endbibitem

\end{thebibliography}

\end{document}